\documentclass[a4paper,reqno]{amsart} 

\usepackage{amssymb}
\usepackage[mathcal]{euscript} 

\usepackage{tikz-cd} 

\usepackage{hyperref}

\newcommand{\bydef}{:=}

\newcommand{\id}{\mathrm{id}}

\newcommand{\trace}{\mathrm{tr}}

\newcommand{\cA}{\mathcal{A}}
\newcommand{\cB}{\mathcal{B}} 
\newcommand{\cC}{\mathcal{C}}
\newcommand{\cD}{\mathcal{D}}

\newcommand{\cI}{\mathcal{I}}
\newcommand{\cJ}{\mathcal{J}}

\newcommand{\cT}{\mathcal{T}}

\newcommand{\frg}{{\mathfrak g}}

\newcommand{\ZZ}{\mathbb{Z}}

\newcommand{\HH}{\mathbb{H}} 
\newcommand{\OO}{\mathbb{O}}
\newcommand{\FF}{\mathbb{F}} 

\DeclareMathOperator{\Hom}{\mathrm{Hom}}
\DeclareMathOperator{\End}{\mathrm{End}}

\DeclareMathOperator{\im}{\mathrm{im}\,}
\DeclareMathOperator{\Aut}{\mathrm{Aut}}

\DeclareMathOperator{\Der}{\mathrm{Der}}

\DeclareMathOperator{\Mat}{\mathrm{Mat}}

\newcommand{\Ad}{\mathrm{Ad}}

\newcommand{\frso}{{\mathfrak{so}}}

\newcommand{\frosp}{{\mathfrak{osp}}}

\newcommand{\tri}{\mathfrak{tri}}

\newcommand{\frel}{\mathfrak{el}}

\newcommand{\Spin}{\mathrm{Spin}}

\providecommand{\espan}[1]{\text{span}\left\{ #1\right\}}

\newcommand{\subo}{_{\bar 0}} 
\newcommand{\subuno}{_{\bar 1}}
\newcommand{\alb}{\mathbb{A}} 

\newcommand{\Cl}{\mathfrak{Cl}} 

\newcommand{\nup}{\textup{n}}

\newcommand{\Repe}{\mathsf{Rep\,C}_3}
\newcommand{\Repp}{\mathsf{Rep\,C}_p}
\newcommand{\Repu}{\mathsf{Rep\,C}_5}
\newcommand{\vect}{\mathsf{Vec}}
\newcommand{\sVec}{\mathsf{sVec}}
\newcommand{\balpha}{\boldsymbol{\alpha}}
\newcommand{\Repa}{\mathsf{Rep}\,\balpha_3}
\newcommand{\Repau}{\mathsf{Rep}\,\balpha_5}
\newcommand{\Repap}{\mathsf{Rep}\,\balpha_p}
\newcommand{\Ver}{\mathsf{Ver}}

\newtheorem{theorem}{Theorem}[section]
\newtheorem{proposition}[theorem]{Proposition}
\newtheorem{lemma}[theorem]{Lemma}
\newtheorem{corollary}[theorem]{Corollary}

\newtheorem{recipe}[theorem]{Recipe}

\theoremstyle{definition}

\theoremstyle{remark} \newtheorem{remark}[theorem]{Remark}
\numberwithin{equation}{section}

\begin{document}

\title[From the Albert algebra to Kac's Jordan superalgebra]%
{From the Albert algebra to Kac's ten-dimensional Jordan superalgebra via tensor categories
in characteristic $5$}


\author[A.~Elduque]{Alberto Elduque} 
\address{Departamento de
Matem\'{a}ticas e Instituto Universitario de Matem\'aticas y
Aplicaciones, Universidad de Zaragoza, 50009 Zaragoza, Spain}
\email{elduque@unizar.es} 
\thanks{A.E. has been supported by grant
MTM2017-83506-C2-1-P (AEI/FEDER, UE) and by grant 
E22\_20R (Gobierno de Arag\'on).}

\author[P.~Etingof]{Pavel Etingof}
\address{Department of Mathematics, Massachusetts Institute of Technology, 
Cambridge, MA 02139}
\email{etingof@math.mit.edu} 
\thanks{P.E.'s work was partially supported by the NSF grant DMS - 1916120}

\author[A.S.~Kannan]{Arun S.~Kannan}
\address{Department of Mathematics, Massachusetts Institute of Technology, 
Cambridge, MA 02139}
\email{akannan@mit.edu} 




\begin{abstract}
Kac's ten-dimensional simple Jordan superalgebra over a field of characteristic $5$ is obtained
from a process of semisimplification, via tensor categories, from the exceptional simple Jordan 
algebra (or Albert algebra), together with a suitable order $5$ automorphism.

This explains McCrimmon's `bizarre result' asserting that, in characteristic $5$, Kac's superalgebra is
a sort of `degree $3$ Jordan superalgebra'.

As an outcome, the exceptional simple Lie superalgebra $\frel(5;5)$, specific of characteristic $5$,
is obtained from the simple Lie algebra of type $E_8$ and an order $5$ automorphism.

In the process, precise recipes to obtain superalgebras from algebras in 
$\Repp$ (or $\Repap$), $p>2$, are given.
\end{abstract}

\maketitle

\section{Introduction}\label{se:Intro}

In \cite{McC05}, Kevin McCrimmon considered the Grassmann evelope of Kac's ten-dimensional
simple Jordan superalgebra $K_{10}$ and obtained, in his own words, \emph{the bizarre
result that in characteristic $5$ (but not otherwise), it is the Jordan algebra over a shaped cubic form 
over $\Gamma_0$}. This means that $K_{10}$ satisfies the super version of the Cayley-Hamilton
equation of degree $3$.

A hint that this could be the case appears in \cite{EO00}, where it is proved that only in
characteristic $5$ there exists an exceptional nine-dimensional \emph{pseudo-composition algebra}, and
this algebra is strongly related to Kac's superalgebra.

McCrimmon's result implies, in characteristic $5$, that $K_{10}$ can be plugged, together
with the algebra of octonions, into a well-known construction by Tits (see \cite{Tits66,Eld11}), 
thus obtaining a Lie superalgebra $\cT(\OO,K_{10})$. This was done in \cite{Eld07}, describing in this way
a new simple Lie superalgebra, denoted by $\frel(5;5)$ in \cite{BGL09}, specific of
characteristic $5$. Its even part is the orthogonal Lie algebra $\frso_{11}$, and the odd part
is the spin module for the even part.

\smallskip

On the other hand, quite recently, the third author \cite{Kan24} obtained this exceptional
Lie superalgebra $\frel(5;5)$, as well as the other known finite-dimensional exceptional
Lie superalgebras with an integer Cartan matrix in characteristic $3$ and $5$, via tensor categories.

In characteristic $3$ one starts with a Lie algebra $\frg$ endowed with a nilpotent derivation $d$
with $d^3=0$. This allows us to see $\frg$ as a Lie algebra in the symmetric tensor category
$\Repa$ of representations of the affine group scheme $\balpha_3:R\mapsto \{a\in R\mid a^3=0\}$.
This category is not semisimple and its semisimplification is the Verlinde category $\Ver_3$, which is 
equivalent to the category of finite-dimensional  vector superspaces $\sVec$. The conclusion is that
$\frg$ \emph{semisimplifies} to a Lie superalgebra.

In characteristic $5$ one starts with a Lie algebra $\frg$ endowed with a nilpotent derivation $d$
with $d^5=0$, so that $\frg$ becomes a Lie algebra in $\Repau$. Its semisimplification gives a 
Lie algebra in the Verlinde category $\Ver_5$. The image of the full tensor subcategory generated by the indecomposable
modules (for the action of $d$) of dimension $1$, $4$ and $5$ under the semisimplification functor 
is equivalent to $\sVec$, as symmetric tensor categories, and again
one obtains Lie superalgebras from suitable Lie algebras. In particular, $\frel(5;5)$ is
obtained in \cite{Kan24} from the exceptional simple Lie algebra of type $E_8$.

\smallskip

The Verlinde category $\Ver_p$ appears too as the semisimplification of the category $\Repp$ of 
finite-dimensional representations of the cyclic group of order $p$ in characteristic $p$. Therefore,
starting with a Lie algebra $\frg$, or any nonassociative algebra in general, endowed with an
automorphism $\sigma$ of order $p$, the algebra $\frg$ becomes an algebra in $\Repp$ and, by
semisimplification, one gets an algebra in $\Ver_p$. 
 
\smallskip

In Section \ref{se:Repu} precise recipes will be given to get a superalgebra from any algebra over
a field of characteristic $p$, endowed with an automorphism of order $p$. Sections \ref{se:Albert}
and \ref{se:Spin} will review the Albert algebra $\alb$, which is the exceptional central simple 
Jordan algebra, and the way that the group $\Spin(8)$ acts on $\alb$ by automorphisms. This will be
used in Section \ref{se:order5} to consider a specific order $5$ automorphism $\sigma$ on the Albert
algebra over a field of characteristic $5$. This automorphism decomposes $\alb$ as a direct sum
$\alb=6L_1\oplus 4L_4\oplus L_5$, where $L_i$ denotes the indecomposable module (for the
action of $\sigma$) of dimension $i$. Hence, by semisimplification, we obtain an algebra in the full
tensor subcategory $\sVec$ of $\Ver_5$. Then, Theorem \ref{th:K10} shows that the corresponding
Jordan superalgebra is Kac's ten-dimensional superalgebra. This explains McCrimmon's ``bizarre
result'' mentioned above (Remark \ref{re:bizarre}). 

Section \ref{se:F4osp} will show that, under the induced action of $\sigma$ on the Lie algebra of
derivations $\Der(\alb)$ (which is the simple Lie algebra of type $F_4$), this Lie algebra semisimplifies
to the Lie superalgebra of derivations of Kac's superalgebra. It must be remarked that the semisimplification of the Lie algebra of derivations of an algebra is not, in general, the Lie superalgebra of derivations of the semisimplification of the algebra (see examples in \cite[Section 4.2]{DES}).

Finally, in Section \ref{se:E8el55} it will be checked that our previous automorphism $\sigma$ on 
the Albert algebra $\alb$ induces naturally an order $5$ automorphism on the simple Lie algebra of 
type $E_8$, obtained by means of Tits construction $\cT(\OO,\alb)$. Its semisimplification
gives the exceptional simple Lie superalgebra $\frel(5;5)$ in a way different from the one
in \cite{Kan24}.

\smallskip

Throughout the paper, \emph{the ground field $\FF$ will be assumed to be algebraically closed and of characteristic $p>2$}. From Section \ref{se:order5} on, this characteristic will be $5$. Unadorned tensor products will always be considered over $\FF$.

\bigskip

\section{From algebras in \texorpdfstring{$\Repp$}{RepCp} to superalgebras}\label{se:Repu}

Let $\mathsf{C}_p$ be the cyclic group of order $p$, and fix a generator $\sigma$. Let $L_i$ be the 
indecomposable module for $\mathsf{C}_p$ of dimension $i$, for $i=1,\ldots,p$. The one-dimensional
(trivial) module $L_1$ will be identified with the ground field $\FF$. In particular, $L_{p-1}$ has a basis
$\{v_0,v_1,\ldots,v_{p-2}\}$ with $\sigma(v_i)=v_i+v_{i+1}$ for all $i=0,\ldots,p-2$, with $v_{p-1}=0$. 
We will consider the element $\delta=\sigma -1$ in the group algebra $\FF \mathsf{C}_p$. Then $\delta$
acts on $L_{p-1}$ by $\delta(v_i)=v_{i+1}$ for all $i=0,\ldots, p-2$.

In \cite{DES} a precise recipe is given to pass from an algebra in $\Repe$ to a superalgebra over fields 
of characteristic $3$. Here this recipe will be extended to an arbitrary characteristic $p>2$. (For $p=2$, the Verlinde category $\Ver_2$ is equivalent to the category of finite-dimensional vector spaces $\vect$.)

The procedure given in \cite{DES} will be followed step-by-step, adding the necessary changes.

\subsection{Semisimplification of \texorpdfstring{$\Repp$}{RepCp}}\label{ss:semiRep}

The category $\Repp$, whose objects are the finite-dimensional representations of the finite group 
$\mathsf{C}_p$ over $\FF$  or, equivalently, of the corresponding
constant group scheme, and whose morphisms are the equivariant homomorphisms, is a symmetric tensor category, with the usual tensor product of vector spaces and the 
braiding given by the usual swap: $X\otimes Y\rightarrow Y\otimes X$, 
$x\otimes y\mapsto y\otimes x$. 

The category $\Repp$ is not  semisimple.  
By the Krull-Schmidt Theorem, any object $\cA$ in $\Repp$ decomposes,
nonuniquely, as
\begin{equation}\label{eq:A1Ap}
\cA=\cA_1\oplus \cA_2\oplus\cdots\oplus \cA_p,
\end{equation}
where $\cA_i$ is a direct sum of copies of $L_i$, $i=1,2,\ldots,p$.

A homomorphism $f\in\Hom_{\Repp}(X,Y)$ is said to be \emph{negligible} if for all homomorphisms $g\in\Hom_{\Repp}(Y,X)$, $\trace(fg)=0$ holds. Denote by 
$\mathrm{N}(X,Y)$ the subspace of negligible homomorphisms in 
$\Hom_{\Repp}(X,Y)$.

\smallskip

Negligible homomorphisms form a \emph{tensor ideal} and this allows us to define the 
\emph{semisimplification} of $\Repp$, which is the Verlinde category $\Ver_p$, whose objects are the objects of $\Repp$, but whose morphisms are given by
\[
\Hom_{\Ver_p}(X,Y)\bydef \Hom_{\Repp}(X,Y)/\mathrm{N}(X,Y).
\]
This is again a symmetric tensor category, with the tensor product in $\Repp$, and the braiding induced by the one in $\Repp$. For convenience we will use the same notation for an object and for its image under semisimplification.

Denote by $[f]$ the class of $f\in\Hom_{\Repp}(X,Y)$ modulo 
$\mathrm{N}(X,Y)$. Note that the identity morphism in $\End_{\Ver_p}(X)$ is 
$[\id_{X}]$, where $\id_{X}$ denotes the identity morphism in $\Repp$ (the identity map). We have thus obtained the \emph{semisimplification} functor:
\begin{equation}\label{eq:S}
\begin{split}
S:\Repp&\longrightarrow \Ver_p\\
 X&\mapsto X \ \text{for objects,}\\
 f&\mapsto [f]\ \text{for morphisms.}
\end{split}
\end{equation}
The semisimplification functor $S$ is $\FF$-linear and braided monoidal (see \cite[Definitions 1.2.3 and 8.1.7]{EGNO}).

Some straightforward consequences of the definitions are recalled here (see, e.g., \cite{EtOs22}):

\begin{itemize}
\item $L_1,\ldots,L_{p-1}$ are simple objects in $\Ver_p$, while $L_p$ is isomorphic to $0$.

\item $\Ver_p$ is semisimple: any object is isomorphic to a direct sum of copies of $L_0,\ldots,L_{p-1}$.

\item $\End_{\Ver_p}(L_i)=\FF[\id_{L_i}]\neq 0$ for $i=1,\ldots,p-1$, $\End_{\Ver_p}(L_p)=0$, and
$\Hom_{\Ver_p}(L_i,L_j)=0$ for $1\leq i\neq j\leq p-1$.

\item $L_1\otimes L_i$ and $L_i\otimes L_1$ are isomorphic to $L_i$, for $i=1,\ldots,p$, both in 
$\Repp$ and in $\Ver_p$.

\item $L_{p-1}\otimes L_{p-1}$ is isomorphic to $L_1$ in $\Ver_p$.
\end{itemize}

An explicit description of an isomorphism $L_1\cong L_{p-1}\otimes L_{p-1}$ in $\Ver_p$ will be needed
later on.

\begin{lemma}\label{le:Lp-1Lp-1L1}
\begin{enumerate}
\item For any $0\leq i,j\leq p-2$ with $i+j\geq p-1$, the element $v_i\otimes v_j$ belongs to the image of the action of $\delta$ in $L_{p-1}\otimes L_{p-1}$.
\item There are scalars $\mu_{ij}\in \FF$, $0\leq i,j\leq p-2$, $i+j\geq p-1$, with
$\mu_{ij}=-\mu_{ji}$ for all $i,j$, such that the element
\begin{equation}\label{eq:w}
w=\sum_{i=0}^{p-2}(-1)^iv_{p-2-i}\otimes v_i\ +
   \sum_{\substack{0\leq i,j\leq p,\\ i+j\geq p-1}}\mu_{ij}v_i\otimes v_j
\end{equation}
is fixed by the action of $\mathsf{C}_p$.
\end{enumerate}
\end{lemma}

\begin{proof}
As in \cite[\S 9]{Wehlau} we may identify $L_{p-1}$ with 
\[
\FF[S]/(S^{p-1})=
\FF[s]=\espan{1,s,\ldots,s^{p-2}},
\] 
where $s$ is the class of the variable $S$ modulo the ideal generated
by $S^{p-1}$. We have $s^{p-1}=0$. The identification takes $v_i$ to $s^i$ for $i=0,\ldots,p-2$. The action of $\sigma$ is then given by $\sigma(s^i)=s^i+s^{i+1}$ for all $i$, so it is given by multiplication by $1+s$. The action of $\delta=\sigma -1$ is given by multiplication by $s$. Then,
$L_{p-1}\otimes L_{p-1}$ can be identified with 
\[
\FF[S]/(S^{p-1})\otimes \FF[T]/(T^{p-1})
\simeq \FF[S,T]/(S^{p-1},T^{p-1})=\FF[s,t],
\] 
so that $v_i\otimes v_j$ corresponds to the monomial $s^it^j$. The action of $\sigma$ is given by multiplication by $(1+s)(1+t)$, and hence the action of
$\delta$ is given by multiplication by $s+t+st$. We have
\[
\cA=\FF[s,t]=\bigoplus_{d=0}^{2p-4}\cA_d,
\]
where $\cA_d=\espan{s^it^j\mid i+j=d}$ is the subspace of homogeneous polynomials of degree $d$ 
in $s$ and $t$. Write $\cA_{\geq d}=\bigoplus_{e\geq d}\cA_e$ and note that we have $\cA_e=0$ 
for $e> 2p-4$.

To prove the first assertion it is enough to show that for any $d\geq p-1$ we have
\begin{equation}\label{eq:Ad}
\cA_d\subseteq \delta\bigl(\cA_{\geq d-1}\bigr).
\end{equation}
Actually, as $\delta(s^{p-2}t^{p-3})=s^{p-2}t^{p-2}$, it follows that $\cA_{2p-4}$ is contained in
$\delta\bigl(\cA_{\geq 2p-5}\bigr)$. Now assume $d\geq p-1$ and 
$\cA_{d+1}\subseteq\delta\bigl(\cA_{\geq d}\bigr)$. We must check that 
$\cA_d\subseteq \delta\bigl(\cA_{\geq d-1}\bigr)$. But for $i=0,\ldots,p-2$, 
\[
\delta(s^it^{d-1-i})=(s+t+st)s^it^{d-1-i}=s^{i+1}t^{d-1-i}+s^it^{d-i}+s^{i+1}t^{d-i}.
\] 
Since
the last summand $s^{i+1}t^{d-i}$ is contained in 
$\cA_{d+1}\subseteq \delta\bigl(\cA_{\geq d}\bigr)$, we get 
\[
s^{i+1}t^{d-1-i}+s^it^{d-i}\in
\delta\bigl(\cA_{\geq d-1}\bigr).
\] 
But $s^0t^d=0$ as $d\geq p-1$, and this gives $s^1t^{d-1}\in\delta\bigl(\cA_{\geq d-1}\bigr)$,
which in turn gives $s^2t^{d-2}\in \delta\bigl(\cA_{\geq d-1}\bigr)$, and so on, thus proving \eqref{eq:Ad}.

\smallskip

Symmetric tensors in $L_{p-1}\otimes L_{p-1}$ correspond to symmetric polynomials in $\cA$:  $f(s,t)=f(t,s)$. Consider the skew-symmetric element
\[
a\bydef \sum_{i=0}^{p-2}(-1)^i s^{p-2-i}t^{i}=s^{p-2}-s^{p-3}t+-\cdots -t^{p-2}\in\cA_{p-2}.
\]
To prove the second assertion, we must check that there exists a skew-symmetric element $b\in\cA_{\geq p-1}$ such that
the skew-symmetric polynomial $w=a+b$ lies in $\ker\delta$.

To check this, note that $(s+t)a$ is trivial and hence we have $\delta(a)=(s+t+st)a=sta\in\cA_{\geq p}
\subseteq \delta\bigl(\cA_{\geq p-1}\bigr)$. Since $a$ is skew-symmetric and $\delta$ preserves this property, there is a skew-symmetric element $b\in\cA_{\geq p-1}$ such that $\delta(a)=-\delta(b)$, and hence $a+b$ lies in $\ker\delta$. 
\end{proof}

\begin{remark}\label{re:easyw}
A particular choice of the element $w$ in \eqref{eq:w} is given by the 
element in $L_{p-1}\otimes L_{p-1}$ that corresponds to the
following element in $\cA$:
\[
f(s,t)=\sum_{i=0}^{p-2}(-1)^i\Bigl(s(1+\tfrac{t}{2})\Bigr)^{p-2-i}
\Bigl(t(1+\tfrac{s}{2})\Bigr)^{i},
\]
because $s+t+st=s(1+\frac{t}{2})+t(1+\frac{s}{2})$, and hence
$(s+t+st)f(s,t)=0$.
\end{remark}

\smallskip

\begin{proposition}\label{pr:Lp-1Lp-1L1}
Any morphism $\lambda:L_{p-1}\otimes L_{p-1}\rightarrow L_1$ in $\Repp$ such that $[\lambda]$ is an isomorphism in $\Ver_p$ satisfies $\lambda(v_0\otimes v_{p-2})\neq 0$. Moreover, after normalizing 
to get $\lambda(v_0\otimes v_{p-2})=1$, the inverse of $[\lambda]$ is $[\lambda']$, where 
$\lambda':\FF\rightarrow L_{p-1}\otimes L_{p-1}$ is the linear map that takes $1$ to the element
$w$ in \eqref{eq:w}. (Note that this is a morphism in $\Repp$ because of item (2) in Lemma \ref{le:Lp-1Lp-1L1}.)
\end{proposition}

\begin{proof}
If $\lambda:L_{p-1}\otimes L_{p-1}\rightarrow L_1$ is a morphism in $\Repp$, it induces a morphism
$L_{p-1}\rightarrow L_{p-1}\cong L_{p-1}^*$. If $[\lambda]$ is an isomorphism, the latter morphism is not negligible, so it is an isomorphism in $\Repp$. In other words, the $\mathsf{C}_p$-invariant bilinear
form induced by $\lambda$ is nondegenerate. Now, if $\lambda(v_0\otimes v_{p-2})$ were $0$,
then we would have
$0=\lambda\bigl(\delta^i(v_0\otimes v_{p-2})\bigr)=\lambda(v_i\otimes v_{p-2})$ for all $i$, a contradiction with the nondegeneracy of $\lambda$.

Now, assume $\lambda(v_0\otimes v_{p-2})=1$.
Since $\lambda$ is trivial on $\delta(L_{p-1}\otimes L_{p-1})$ and
 $v_i\otimes v_j\in \delta(L_{p-1}\otimes L_{p-1})$ for $i+j\geq p-1$ by Lemma \ref{le:Lp-1Lp-1L1}, we have
\[
\lambda(w)=\sum_{i=0}^{p-2}(-1)^i\lambda(v_{p-2-i}\otimes v_i).
\]
Also, we get
\[
\begin{split}
0&=\lambda\bigl(\delta(v_{p-3-i}\otimes v_i)\bigr) \\
 &=\lambda(v_{p-2-i}\otimes v_i)+\lambda(v_{p-3-i}\otimes v_{i+1})
                     +\lambda(v_{p-2-i}\otimes v_{i+1})\\
 &=\lambda(v_{p-2-i}\otimes v_i)+\lambda(v_{p-3-i}\otimes v_{i+1})\quad
 \text{because $(p-2-i)+(i+1)=p-1$}.
\end{split}
\]
Hence 
\begin{equation}\label{eq:lambdaii+1}
\lambda(v_{p-2-i}\otimes v_i)=-\lambda(v_{p-3-i}\otimes v_{i+1})
\end{equation} 
and
\[
\begin{split}
\lambda(w)&=\lambda(v_{p-2}\otimes v_0)-\lambda(v_{p-3}\otimes v_1)+-
         \cdots-\lambda(v_0\otimes v_{p-2})\\
    &=-(p-1)\lambda(v_0\otimes v_{p-2})=1.
\end{split}
\]
We conclude that $\lambda\circ \lambda'=\id$, and hence $[\lambda]^{-1}=[\lambda']$, as required.
\end{proof}

\begin{remark}\label{re:walphap}
Had we considered $\Repap$ instead of $\Repp$, the element $w$
to consider in \eqref{eq:w} would be simpler. An object in $\Repap$ is
a left module for $\FF[t]$, $t^p=0$. The action of $t$ on a tensor 
product being given by $t(a\otimes b)=t(a)\otimes b+a\otimes t(b)$.
Then the `easier' morphism in $\Repap$ given by:
\[
\begin{split}
\lambda':\FF&\rightarrow L_{p-1}\otimes L_{p-1}\\
1&\mapsto \sum_{i=0}^{p-2}(-1)^iv_{p-2-i}\otimes v_i,
\end{split}
\]
induces an isomorphism $[\lambda']$ in $\Ver_p$. (Here 
$\{v_0,\ldots,v_{p-2}\}$ is a basis of $L_{p-1}$ with 
$t(v_i)=v_{i+1}$, $0\leq i\leq p-3$, $t(v_{p-2})=0$.)

Anyway, the order $5$ automorphism of the Albert algebra over a field of characteristic $5$ considered in Section \ref{se:order5} is quite natural and easy to deal with. Hence, we will stick to $\Repp$.
\end{remark}

\smallskip

As in \cite{DES}, the braiding in $\Ver_p$, for objects $X,Y$, is given by $[c_{X,Y}]$, 
where $c_{X,Y}$ is the braiding in $\Repp$ (i.e., the swap $x\otimes y\mapsto y\otimes x$). Then, identifying 
$L_1\otimes L_1\simeq L_1$,
$L_1\otimes L_{p-1}\simeq L_{p-1}\simeq L_{p-1}\otimes L_1$, and $L_{p-1}\otimes L_{p-1}\simeq L_1$ as above, we have $[c_{L_1,L_1}]=[\id_{L_1}]$, 
$[c_{L_1,L_{p-1}}]=[\id_{L_{p-1}}]=[c_{L_{p-1},L_1}]$, but $[c_{L_{p-1},L_{p-1}}]=-[\id_{L_1}]$, because the element $w$ is skew-symmetric.


\subsection{Equivalence of \texorpdfstring{$\sVec$}{sVec} and the full tensor subcategory
of \texorpdfstring{$\Ver_p$}{Verp} generated by 
\texorpdfstring{$L_1$}{L1} and \texorpdfstring{$L_{p-1}$}{Lp-1}}\label{ss:Ver_sVec}

This equivalence is well known, but we are interested in  concrete formulas for this equivalence.

The objects of the category $\sVec$ of vector superspaces (over our ground field 
$\FF$) are the $\ZZ/2$-graded finite-dimensional vector spaces $X=X\subo\oplus X\subuno$, and the morphisms $f:X\rightarrow Y$ are the linear maps preserving this grading: $f(X\subo)\subseteq Y\subo$, $f(X\subuno)\subseteq Y\subuno$. We will write 
$f=f\subo\oplus f\subuno$, with $f_{\bar a}:X_{\bar a}\rightarrow Y_{\bar a}$ given by the restriction of $f$, $a=0,1$. This is a symmetric tensor category, with the braiding given by the \emph{signed swap}: 
\[
c_{X,Y}:X\otimes Y\rightarrow Y\otimes X, \quad
x\otimes y\mapsto (-1)^{\lvert x\rvert\lvert y\rvert}y\otimes x,
\] 
for homogeneous elements $x,y$, where $(-1)^{\lvert x\rvert\lvert y\rvert}$ is 
$-1$ if both $x$ and $y$ are odd, and it is $1$ otherwise.

The $\FF$-linear functor given on objects and morphisms by
\begin{equation}\label{eq:F}
\begin{split}
F:\sVec&\longrightarrow \Ver_p\\
X\subo\oplus X\subuno &\mapsto X\subo\oplus(X\subuno\otimes L_{p-1})\\
f\subo\oplus f\subuno&\mapsto [f\subo\oplus(f\subuno\otimes\id_{L_{p-1}})],
\end{split}
\end{equation}
provides an equivalence of symmetric tensor categories between $\sVec$ and the full tensor subcategory of $\Ver_p$
generated by $L_1$ and $L_{p-1}$. Here the action of $\mathsf{C}_p$ on 
$X\subo\oplus (X\subuno\otimes L_{p-1})$ is given by the action on $L_{p-1}$. That is, 
$X\subo$ is a trivial module for $\mathsf{C}_p$, while 
$\sigma(x\subuno\otimes v)\bydef x\subuno\otimes \sigma(v)$, for all 
$x\subuno\in X\subuno$ and $v\in L_{p-1}$.

The functor $F$ is a monoidal functor with natural isomorphism 
$J:F(\cdot)\otimes F(\cdot)\rightarrow F(\cdot\otimes\cdot)$ given by 
$J_{X,Y}=[j_{X,Y}]$, where $j_{X,Y}$ is the morphism in $\Repp$ defined as follows,
for $X=X\subo\oplus X\subuno$ and $Y=Y\subo\oplus Y\subuno$:
\begin{equation}\label{eq:J}
\begin{split}
j_{X,Y}:\Bigl(X\subo\oplus(X\subuno&\otimes L_{p-1})\Bigr)\otimes 
\Bigl(Y\subo\oplus(Y\subuno\otimes L_{p-1})\Bigr)\\
&\longrightarrow(X\subo\otimes Y\subo\oplus X\subuno\otimes Y\subuno)\oplus
\bigl((X\subo\otimes Y\subuno\oplus X\subuno\otimes Y\subo)\otimes L_{p-1}\bigr)\\
x\subo\otimes y\subo&\mapsto x\subo\otimes y\subo,\\
x\subo\otimes (y\subuno\otimes v)&\mapsto (x\subo\otimes y\subuno)\otimes v,\\
(x\subuno\otimes v)\otimes y\subo&\mapsto (x\subuno\otimes y\subo)\otimes v,\\
(x\subuno\otimes u)\otimes (y\subuno\otimes v)&\mapsto 
\lambda(u\otimes v)x\subuno\otimes y\subuno,
\end{split}
\end{equation}
for all $x\subo\in X\subo$, $x\subuno\in X\subuno$, $y\subo\in Y\subo$, $y\subuno\in Y\subuno$, and $u,v\in L_{p-1}$, where $\lambda$ is any fixed morphism as in Proposition \ref{pr:Lp-1Lp-1L1} with 
$\lambda(v_0\otimes v_{p-2})=1$.

\smallskip

The inverse of $J_{X,Y}$ is $J_{X,Y}^{-1}=[j'_{X,Y}]$, where $j'_{X,Y}$ is defined as follows:

\bigskip

\pagebreak

\begin{equation}\label{eq:J'}
\begin{split}
j'_{X,Y}:(X\subo\otimes Y\subo\oplus X\subuno\otimes Y\subuno)&\oplus
\bigl((X\subo\otimes Y\subuno\oplus X\subuno\otimes Y\subo)\otimes L_{p-1}\bigr)\\
&\longrightarrow
\Bigl(X\subo\oplus(X\subuno\otimes L_{p-1})\Bigr)\otimes 
\Bigl(Y\subo\oplus(Y\subuno\otimes L_{p-1})\Bigr)
\\
x\subo\otimes y\subo&\mapsto x\subo\otimes y\subo,\\
(x\subo\otimes y\subuno)\otimes v &\mapsto x\subo\otimes (y\subuno\otimes v),\\
(x\subuno\otimes y\subo)\otimes v&\mapsto (x\subuno\otimes v)\otimes y\subo,\\
x\subuno\otimes y\subuno&\mapsto 
         \sum_{i=0}^{p-2}(-1)^i(x\subuno\otimes v_{p-2-i})\otimes (y\subuno\otimes v_i)\\
 &\qquad\qquad + \sum_{\substack{0\leq i,j\leq p,\\ i+j\geq p-1}}\mu_{ij}(x\subuno\otimes v_i)\otimes (y\subuno\otimes v_j).
\end{split}
\end{equation}
In other words, $j_{X,Y}'(x\subuno\otimes y\subuno)=c_{23}(x\subuno\otimes y\subuno\otimes w)$, with 
$w$ in \eqref{eq:w}, and where $c_{23}(a\otimes b\otimes c\otimes d)=a\otimes c\otimes b\otimes d$.

Note that $F$ preserves the braiding too. In other words, the following diagram is commutative for all $X,Y$:
\[
\begin{tikzcd}
F(X)\otimes F(Y)\arrow[rr, "\hbox{$[\text{`swap'}]$}"]\arrow[d, "J_{X,Y}"']&&
F(Y)\otimes F(X)\arrow[d, "J_{Y,X}"]\\
F(X\otimes Y)\arrow[rr, "F(\text{`signed swap'})"]&&F(Y\otimes X)
\end{tikzcd}
\]
Therefore, the functor $F$ is braided and monoidal.

\smallskip

If $\cA$ is an object in $\Repp$, fix a decomposition $\cA=\cA_1\oplus\cdots\oplus \cA_p$, where 
$\cA_i$ is a direct sum of copies of $L_i$ for $i=1,\ldots,p$. Write $A\subo=\cA_1$, and fix a 
subspace $A\subuno$ of $\cA_{p-1}$ such that $\cA_{p-1}=A\subuno\oplus \delta(\cA_{p-1})$. 
(That is, $A\subuno$ is a subspace spanned by the `heads' of the Jordan blocks of the action of 
$\sigma$ or $\delta$ on 
$\cA_{p-1}$.)

Then, $A\bydef A\subo\oplus A\subuno$ is an object in $\sVec$. Consider the morphism in $\Repp$:
\begin{equation}\label{eq:iotaA}
\begin{split}
\iota_{\cA}: F(A)=A\subo\oplus(A\subuno\otimes L_{p-1})&\longrightarrow \cA\\
                                                  a\subo\quad&\mapsto\ a\subo\in\cA_1,\\
                                                  a\subuno\otimes v_i\ &\mapsto\  \delta^i(a\subuno)\in\cA_{p-1}.
\end{split}
\end{equation}
The image of $\iota_\cA$ is $\cA'\bydef\cA_1\oplus \cA_{p-1}$.

\begin{remark}\label{re:iota_iso}
If $\cA=\cA_1\oplus\cA_{p-1}\oplus\cA_p$ (and this is always the case for $p=3$), then $[\iota_\cA]$ is an isomorphism in $\Ver_p$.
\end{remark}

\bigskip

\subsection{Recipe to get superalgebras from algebras in 
\texorpdfstring{$\Repp$}{RepCp}}

Given a morphism $m:A\otimes B\rightarrow C$ in $\sVec$, the composition
\[
F(A)\otimes F(B)\xrightarrow{J_{A,B}}F(A\otimes B)\xrightarrow{F(m)} F(C)
\]
gives a morphism $F(A)\otimes F(B)\rightarrow F(C)$ in $\Ver_p$. In particular, with $A=B=C$, given an algebra $(A,m)$ in $\sVec$, $F(A)$ is an algebra in 
$\Ver_p$ with multiplication given by the composition
\[
F(A)\otimes F(A)\xrightarrow{J_{A,A}}F(A\otimes A)\xrightarrow{F(m)} F(A).
\]

Now, given a morphism $\mu:\cA\otimes \cB\rightarrow \cC$ in $\Repp$, decompose 
$\cA=\bigoplus_{i=1}^{p}\cA_i$, $\cB=\bigoplus_{i=1}^{p}\cB_i$, and 
$\cC=\bigoplus_{i=1}^{p}\cC_i$ as above, and consider the associated objects
$A=A\subo\oplus A\subuno$, $B=B\subo\oplus B\subuno$, $C=C\subo\oplus C\subuno$ in $\sVec$. 
Our goal
is to find a morphism $m:A\otimes B\rightarrow C$ in $\sVec$ such that the diagram
\begin{equation}\label{eq:FABCmuFm}
\begin{tikzcd}
F(A)\otimes F(B)\arrow[r, "J_{A,B}"]\arrow[d, "\hbox{$[\iota_{\cA}\otimes\iota_{\cB}]$}"']&
F(A\otimes B)\arrow[r, "F(m)"]&F(C)\arrow[d, "\hbox{$[\iota_{\cC}]$}"]\\
\cA\otimes\cB\arrow[rr, "\hbox{$[\mu]$}"]&&\cC
\end{tikzcd}
\end{equation}
is commutative.

This was done for $p=3$ in \cite{DES}, where $[\iota_{\cA}],[\iota_{\cB}],[\iota_{\cC}]$ are isomorphisms. A variation of the arguments there works. To begin, consider the inclusion map
$i_{\cA}:\cA'\bydef \cA_1\oplus \cA_{p-1}\hookrightarrow\cA$, and similarly $i_{\cB}$ and $i_{\cC}$. The linear map
\[
\mu':\cA'\otimes\cB'\longrightarrow \cC'
\]
given by 
\[
\mu'(x\otimes y)\bydef\begin{cases}
\mathrm{proj}_{\cC_1} \mu(x\otimes y)&\text{for $x\in\cA_1, y\in\cB_1$ or 
                                $x\in\cA_{p-1}, y\in\cB_{p-1}$,}\\
\mathrm{proj}_{\cC_{p-1}} \mu(x\otimes y)&\text{for $x\in\cA_1,y\in\cB_{p-1}$, or
                              $x\in\cA_{p-1},y\in\cB_1$,}
\end{cases}
\]
(projections relative to the decomposition $\cC=\cC_1\oplus\cdots\oplus \cC_p$), is also a morphism  
in $\Repp$ and the diagram 
\begin{equation}\label{eq:mu'}
\begin{tikzcd}
\cA'\otimes\cB'\arrow[rr, "\hbox{$[\mu']$}"]
\arrow[d, hook]&& \cC'\arrow[d, hook]\\
\cA\otimes\cB\arrow[rr, "\hbox{$[\mu]$}"]&&\cC
\end{tikzcd}
\end{equation}
is commutative.

In particular, if $(\cA,\mu)$ is an algebra in $\Repp$, then $(\cA',\mu')$ is another algebra in $\Repp$,
and $(\cA',[\mu'])$ is a subalgebra of the algebra $(\cA,[\mu])$ in $\Ver_p$.

\smallskip

Let $\mu:\cA\otimes\cB\rightarrow\cC$ be a morphism in $\Repp$. Decompose $\cA=\bigoplus_{i=1}^{p-2}\cA_i$, where $\cA_i$ is a direct sum of copies of $L_i$ for the action of $\mathsf{C}_p$, and similarly for $\cB$ and $\cC$. Write $A\subo=\cA_1$ and fix a subspace $A\subuno$ of $\cA_{p-1}$ with $\cA_{p-1}=A\subuno\oplus\delta(\cA_{p-1})$, and do the same for $\cB$ and $\cC$. In particular, we have
\begin{equation}\label{eq:Csplitting}
\cC=C\subo\oplus\cC_2\oplus\cdots\oplus \cC_{p-2}\oplus C\subuno\oplus\delta(\cC_{p-1})\oplus\cC_p.
\end{equation}

\begin{recipe}\label{recipe1} Take projections relative to the decomposition
\eqref{eq:Csplitting}, and define the linear map $m:A\otimes B\rightarrow C$  as follows:
\[
\begin{split}
m(x\subo\otimes y\subo) &=\mathrm{proj}_{C\subo} \mu(x\subo\otimes y\subo)\\
m(x\subo\otimes y\subuno)&=\mathrm{proj}_{C\subuno} \mu(x\subo\otimes y\subuno)\\
m(x\subuno\otimes y\subo)&=\mathrm{proj}_{C\subuno} \mu(x\subuno\otimes y\subo)\\
m(x\subuno\otimes y\subuno)&=\mathrm{proj}_{C\subo} 
                    \mu\bigl(x\subuno\otimes\delta^{p-2}(y\subuno)\bigr)
\end{split}
\]
for all $x\subo\in A\subo$, $y\subo\in B\subo$ and $x\subuno\in A\subuno$,
$y\subuno\in B\subuno$. 
The linear map $m$ is a morphism in the category $\sVec$.
\end{recipe}

This morphism $m$ is what we need to make the diagram in \eqref{eq:FABCmuFm} commutative.

\begin{theorem}\label{th:main}
Let $\mu:\cA\otimes\cB\rightarrow\cC$ be a homomorphism in $\Repp$. Pick
decompositions of $\cA$, $\cB$, and $\cC$ as in  \eqref{eq:Csplitting}. Define a morphism 
$m:A\otimes B\rightarrow C$
in $\sVec$ by means of Recipe \ref{recipe1}. Then, with $\iota_A,\iota_B,\iota_C$ 
as in 
\eqref{eq:iotaA}, the diagram \eqref{eq:FABCmuFm} is commutative.
\end{theorem}

\begin{proof}
Because of the commutativity of the diagram in \eqref{eq:mu'}, it is enough to prove that 
the diagram (in $\Ver_p$)
\begin{equation}\label{eq:FABCmuFm'}
\begin{tikzcd}
F(A)\otimes F(B)\arrow[r, "J_{A,B}"]\arrow[d, "\hbox{$[\iota_A\otimes\iota_B]$}"']&
F(A\otimes B)\arrow[r, "F(m)"]&F(C)\arrow[d, "\hbox{$[\iota_C]$}"]\\
\cA'\otimes\cB'\arrow[rr, "\hbox{$[\mu']$}"]&&\cC'
\end{tikzcd}
\end{equation}
is commutative. (Here we use the same notation $\iota_A$ to denote the isomorphism
$F(A)\simeq \cA'$ in $\Repp$ induced by the original $\iota_A:F(A)\rightarrow\cA$ in
\eqref{eq:iotaA}.)

Using the inverse of $J_{A,B}$ (see \eqref{eq:J'}), this is equivalent to checking in the following diagram
 in $\Repp$:
\[
\begin{tikzcd}
F(A)\otimes F(B)\arrow[d, "\iota_A\otimes\iota_B"']&
F(A\otimes B)\arrow[l, "j'_{A,B}"']\arrow[rr, 
"m\subo\oplus(m\subuno\otimes\id_{V_1}\hspace*{-2pt})"]&&F(C)\arrow[d, "\iota_C"]\\
\cA'\otimes\cB'\arrow[rrr, "\mu'"]&&&\cC'
\end{tikzcd}
\]
that the difference $\Phi\bydef 
\iota_C\circ\bigl(m\subo\oplus(m\subuno\otimes\id_{V_1})\bigr)
- \mu'\circ(\iota_A\otimes\iota_B)\circ j'_{A,B}$ is negligible.

For $x\subo\in A\subo,y\subo\in B\subo$ we get
\[
\begin{split}
&x\subo\otimes y\subo\xrightarrow{j'_{A,B}}x\subo\otimes y\subo
\xrightarrow{\iota_A\otimes\iota_B} x\subo\otimes y\subo
\xrightarrow{\mu'}\mu'(x\subo\otimes y\subo),\\
&x\subo\otimes y\subo\xrightarrow{m\subo}m(x\subo\otimes y\subo)
 =\mu'(x\subo\otimes y\subo)
\xrightarrow{\iota_{\cC}}\mu'(x\subo\otimes y\subo),
\end{split}
\]
so that $\Phi$ is trivial on $A\subo\otimes B\subo$. 

For 
$x\subuno\in A\subuno$, $y\subuno\in B\subuno$ we get that both
$\espan{x\subuno,\delta(x\subuno),\ldots,\delta^{p-2}(x\subuno)}$ and
$\espan{y\subuno,\delta(y\subuno),\ldots,\delta^{p-2}(y\subuno)}$ are isomorphic in $\Repp$ to 
$L_{p-1}$. Equation \eqref{eq:Ad} shows that for $i+j\geq p-1$, 
$\delta^i(x\subuno)\otimes \delta^j(y\subuno)$ is in the image of the action of $\delta$ on
the tensor product
\[
\espan{x\subuno,\delta(x\subuno),\ldots,\delta^{p-2}(x\subuno)}\otimes
\espan{y\subuno,\delta(y\subuno),\ldots,\delta^{p-2}(y\subuno)},
\] 
say 
$\delta^i(x\subuno)\otimes \delta^j(y\subuno)=\delta(Z)$, and hence 
\[
\mu'\bigl(\delta^i(x\subuno)\otimes \delta^j(y\subuno)\bigr)
=\mu'\bigl(\delta(Z)\bigr)=\delta\bigl(\mu'(Z)\bigr)=0,
\] 
because $\mu'$ is a morphism in $\Repp$, and
$\mu'(Z)$ lies in $\cC_1$, so it is annihilated by $\delta$.

Also, the arguments leading to \eqref{eq:lambdaii+1} give
\[
\mu'\bigl(\delta^{p-2-i}(x\subuno)\otimes\delta^{i}(y\subuno)\bigr)
   =-\mu'\bigl(\delta^{p-3-i}(x\subuno)\otimes\delta^{i+1}(y\subuno)\bigr).
\]
Then, we get
\begin{equation}\label{eq:wx1y1}
\begin{split}
x\subuno\otimes y\subuno&\xrightarrow{j'_{A,B}}
 \sum_{i=0}^{p-2}(-1)^i(x\subuno\otimes v_{p-2-i})\otimes(y\subuno\otimes v_i) +
   \sum_{\substack{0\leq i,j\leq p,\\ i+j\geq p-1}}\mu_{ij}(x\subuno\otimes v_i)\otimes         
           (y\subuno\otimes v_j)\\
&\xrightarrow{\iota_A\otimes \iota_B}
  \sum_{i=0}^{p-2}(-1)^i\delta^{p-2-i}(x\subuno)\otimes \delta^i(y\subuno) +
   \sum_{\substack{0\leq i,j\leq p,\\ i+j\geq p-1}}\mu_{ij}\delta^i(x\subuno)\otimes 
             \delta^j(y\subuno)\\
&\xrightarrow{\ \mu'\ } 
    \sum_{i=0}^{p-2}(-1)^i\mu'\bigl(\delta^{p-2-i}(x\subuno)\otimes \delta^i(y\subuno)\bigr)\\
&\qquad\qquad\qquad\qquad\qquad
   = -(p-1)\mu'\bigl(x\subuno\otimes\delta^{p-2}(y\subuno)\bigr)
   =\mu'\bigl(x\subuno\otimes\delta^{p-2}(y\subuno)\bigr)\\[8pt]
x\subuno\otimes y\subuno&\xrightarrow{m\subo}m(x\subuno\otimes y\subuno)
   =\mu'\bigl(x\subuno\otimes \delta^{p-2}(y\subuno)\bigr)
\xrightarrow{\iota_{\cC}}\mu'\bigl(x\subuno\otimes \delta^{p-2}(y\subuno)\bigr),
\end{split}
\end{equation}
and $\Phi$ is trivial too on $A\subuno\otimes B\subuno$. 

Now, for $x\subo\in A\subo$
and $y\subuno\in B\subuno$, and for $i=0,\ldots,p-2$, we get
\[
\begin{split}
&(x\subo\otimes y\subuno)\otimes v_i
 \xrightarrow{j'_{A,B}} x\subo\otimes (y\subuno\otimes v_i)
                    \xrightarrow{\iota_A\otimes\iota_B} x\subo\otimes\delta^i(y\subuno)
 \xrightarrow{\mu'} \mu'\bigl(x\subo\otimes\delta^i(y\subuno) \bigr),
 \\[4pt]
&(x\subo\otimes y\subuno)\otimes v_i
 \xrightarrow{m\subuno\otimes\id_{L_{p-1}}} m(x\subo\otimes y\subuno)\otimes v_i
   \xrightarrow{\ \iota_{\cC}\ } \delta^i\bigl(m(x\subo\otimes y\subuno)\bigr)
\end{split}
\]
Hence we have
\[
\begin{split}
\Phi\bigl((x\subo\otimes y\subuno)\otimes v_i\bigr)
& = \delta^i\bigl(m(x\subo\otimes y\subuno)\bigr)
     -\mu'\bigl(x\subo\otimes\delta^i(y\subuno) \bigr)\\
& =\delta^i\bigl(m(x\subo\otimes y\subuno)-\mu'(x\subo\otimes y\subuno)\bigr)
\\
& =\delta^i\bigl(\mathrm{proj}_{C\subuno}(\mu(x\subo\otimes y\subuno))-
   \mathrm{proj}_{\cC_{p-1}}(\mu(x\subo\otimes y\subuno))\bigr)\\
& =-\delta^i\bigl(\mathrm{proj}_{\delta(\cC_{p-1})}(\mu(x\subo\otimes y\subuno))\bigr).
\end{split}
\]
It follows that the restriction $\Phi\vert_{(A\subo\otimes B\subuno)\otimes L_{p-1}}$,
takes $(A\subo\otimes B\subuno)\otimes L_{p-1}$, which is a direct sum of copies of
$L_{p-1}$, to $\delta(\cC_{p-1})$, which is a direct sum of copies of $L_{p-2}$, and hence it is negligible. In the same vein, the restriction 
$\Phi\vert_{(A\subuno\otimes B\subo)\otimes L_{p-1}}$ is negligible.

We conclude that $\Phi$ is negligible, as required.
\end{proof}

In particular, if $(\cA,\mu)$ is an algebra in $\Repp$, with $\mu(x\otimes y)=xy$ for all $x,y$, and we fix a decomposition   $\cA=A\subo\oplus \cA_2\oplus\cdots\cA_{p-2}\oplus A\subuno\oplus \delta(\cA_{p-1})\oplus\cA_p$ as in \eqref{eq:Csplitting}, Recipe \ref{recipe1} becomes the following one.

\begin{recipe}\label{recipe} Take projections relative to the decomposition above, and define a multiplication $m$ 
($m(x\otimes y)\bydef x\diamond y$) on $A\bydef A\subo\oplus A\subuno$ as follows:
\[
\begin{split}
x\subo\diamond y\subo &=\mathrm{proj}_{A\subo} (x\subo y\subo)\\
x\subo\diamond y\subuno&=\mathrm{proj}_{A\subuno} (x\subo y\subuno)\\
x\subuno\diamond y\subo&=\mathrm{proj}_{A\subuno} (x\subuno y\subo)\\
x\subuno\diamond y\subuno&=\mathrm{proj}_{A\subo} 
                    \bigl(x\subuno\delta^{p-2}(y\subuno)\bigr)
\end{split}
\]
for all $x\subo,y\subo\in A\subo$ and $x\subuno,y\subuno\in A\subuno$. 

The algebra $(A,m)$ is an algebra in $\sVec$ (a superalgebra).
\end{recipe}

In this case, Theorem \ref{th:main} restricts to the following Corollary:

\begin{corollary}\label{co:main}
Let $(\cA,\mu)$ be an algebra in $\Repp$, with $\mu(x\otimes y)=xy$ for all $x,y$. Fix a decomposition
$\cA=A\subo\oplus \cA_2\oplus\cdots\cA_{p-2}\oplus A\subuno\oplus \delta(\cA_{p-1})\oplus\cA_p$ 
as in \eqref{eq:Csplitting}. Define a multiplication $m$ in $A=A\subo\oplus A\subuno$ by means of Recipe \ref{recipe}.
Then, the algebra  $\bigl(F(A),F(m)\circ J_{A,A}\bigr)$ in $\Ver_p$ is isomorphic to the subalgebra
$(\cA',[\mu'])$, with $\mu'$ in \eqref{eq:mu'} of the algebra $(\cA,[\mu])$
\end{corollary}

In case $\cA=\cA_1\oplus\cA_{p-1}\oplus \cA_p$, the morphism $[\iota_\cA]$ is an isomorphism,
and the algebra $\bigl(F(A),F(m)\circ J_{A,A}\bigr)$ in $\Ver_p$ is isomorphic to $(\cA,[\mu])$.
In this situation, we will say that the algebra $\cA$ \emph{semisimplifies} to the superalgebra 
$A=A\subo\oplus A\subuno$.

\begin{remark}\label{re:invariant}
The reader may feel uncomfortable with the fact that Recipe \ref{recipe} (or Recipe \ref{recipe1}) depends on the decomposition $\cA=A\subo\oplus\cA_2\oplus\cdots\oplus\cA_{p-2}\oplus A\subuno
\oplus\delta(\cA_{p-1})\oplus\cA_p$, which is not canonical. A canonical version of Recipe \ref{recipe} is given in Appendix \ref{se:App1}. A similar canonical version of Recipe \ref{recipe1} is deduced in the same way.
\end{remark}

\begin{remark}\label{re:alphap}
Recipes \ref{recipe1} and \ref{recipe} remain valid if $\Repap$ is 
considered instead of $\Repp$. The arguments are a bit simpler, because of Remark \ref{re:walphap}.
\end{remark}

\bigskip

\section{The Albert algebra}\label{se:Albert}

Let $\OO$ be the Cayley algebra over $\FF$. It has a basis $\{e_i\}_{i=0}^7$ with multiplication given by the following relations:
$e_0=1$ is the unity element, $e_i^2=-e_0$ for any $i=1,\ldots,7$, and
\[
e_ie_{i+1}=e_{i+3}=-e_{i+1}e_i\quad\text{and cyclically on $i,i+1,i+3$, for $i=1,\ldots,7$,}
\]
where the indices are taken modulo $7$. The canonical involution $\nu:x\mapsto \bar x$ is given by 
$\nu(e_0)=e_0$ and
$\nu(e_i)=-e_i$, for $i=1,\ldots,7$, and the multiplicative norm is determined by $\nup(e_i)=1$ for any
$i$ and $\nup(e_i,e_j)\,\bigl(\bydef\nup(e_i+e_j)-\nup(e_i)-\nup(e_j)\bigr)=0$ for $i\neq j$. That is, the given
basis is orthonormal.

Then, the Albert algebra is the algebra of $3\times 3$ Hermitian matrices over the Cayley algebra: 
$\alb=H_3(\OO)$, with 
multiplication given by
\[
X\cdot Y=\frac{1}{2}(XY+YX).
\]
Then we have $\alb=\FF E_1\oplus \FF E_2\oplus \FF E_3\oplus \iota_1(\OO)\oplus\iota_2(\OO)\oplus\iota_3(\OO)$,
with
\begin{equation}\label{eq:Esiotas}
\begin{aligned}
E_1&=\begin{pmatrix} 1&0&0\\ 0&0&0\\ 0&0&0\end{pmatrix}, &
E_2&=\begin{pmatrix} 0&0&0\\ 0&1&0\\ 0&0&0\end{pmatrix}, &
E_3&=\begin{pmatrix} 0&0&0\\ 0&0&0\\ 0&0&1\end{pmatrix}, \\
\iota_1(a)&=2\begin{pmatrix} 0&0&0\\ 0&0&\bar a\\ 0&a&0\end{pmatrix}, &
\iota_2(a)&=2\begin{pmatrix} 0&0&a\\ 0&0&0\\ \bar a&0&0\end{pmatrix}, &
\iota_3(a)&=2\begin{pmatrix} 0&\bar a&0\\ a&0&0\\ 0&0&0\end{pmatrix}, 
\end{aligned}
\end{equation}
for any $a\in\OO$. The $E_i$'s are idempotents, and their sum is the unity $1$ of $\alb$. The other products
are as follows (see \cite[(5.2)]{EKmon}):
\[
\begin{split}
&E_i\cdot\iota_i(a)=0,\quad E_{i+1}\cdot\iota_i(a)=\frac{1}{2}\iota_i(a)=E_{i+2}\cdot\iota_i(a),\\
&\iota_i(a)\cdot\iota_{i+1}(b)=\iota_{i+2}(a\bullet b),\quad 
          \iota_i(a)\cdot\iota_i(b)=2\nup(a,b)(E_{i+1}+E_{i+2}),
\end{split}
\]
for any $a,b\in\OO$ with $i=1,2,3$ taken modulo $3$, and where $a\bullet b\bydef \bar a\,\bar b$ is the
\emph{para-Hurwitz} product on $\OO$. 

\smallskip

The Albert algebra is naturally graded by $\bigl(\ZZ/2\bigr)^2$ with:
\begin{equation}\label{eq:Z22grading}
\alb_{(\bar 0,\bar 0)}=\FF E_1\oplus \FF E_2\oplus \FF E_3,\quad
\alb_{(\bar 0,\bar 1)}=\iota_1(\OO),\quad
\alb_{(\bar 1,\bar 0)}=\iota_2(\OO),\quad
\alb_{(\bar 1,\bar 1)}=\iota_3(\OO).
\end{equation}

\bigskip

\section{The spin group}\label{se:Spin}

The linear map
\begin{equation}\label{eq:Phi1}
\begin{split}
\OO&\longrightarrow \End_\FF(\OO\oplus\OO)\\
x&\ \mapsto\quad \begin{pmatrix} 0&l_x\\ r_x&0\end{pmatrix},
\end{split}
\end{equation}
where $l_x(y)=r_y(x)=x\bullet y\,(=\bar x\bar y)$, induces an algebra isomorphism
\[
\Phi:\Cl(\OO,\nup)\longrightarrow \End_\FF(\OO\oplus\OO).
\]
 To avoid confusion, we will denote by $x\cdot y$ the multiplication in the Clifford algebra 
 $\Cl(\OO,\nup)$.

For any even element $u\in\Cl\subo(\OO,\nup)$, we have 
\begin{equation}\label{eq:Phi2}
\Phi(u)=\begin{pmatrix}\rho_u^-&0\\ 0&\rho_u^+\end{pmatrix},
\end{equation}
for suitable endomorphisms $\rho^{\pm}_u\in\End_\FF(\OO)$.

Recall that the spin group associated to the quadratic pair $(\OO,\nup)$ is the group
\[
\Spin(\OO,\nup)=\{u\in\Cl\subo(\OO,\nup)\mid u\cdot \tau(u)=1\ 
    \text{and}\ u\cdot \OO\cdot u^{-1}=\OO\},
\]
where $\tau$ is the standard involution of $\Cl(\OO,\nup)$, whose restriction to $\OO$ is the identity.

We have the following group isomorphism (see, e.g., \cite[Theorem 5.5]{EKmon}):
\begin{equation}\label{eq:Spin8RelatedTriples}
\begin{split}
\Spin(\OO,\nup)&\longrightarrow \{(\varphi_1,\varphi_2,\varphi_3)\in \mathrm{SO}(\OO,\nup)^3\mid \varphi_1(x\bullet y)=\varphi_2(x)\bullet \varphi_3(y)\ \forall x,y\in\OO\}\\
 u\quad &\ \mapsto \qquad \bigl(\chi_u,\rho^+_u,\rho^-_u\bigr),
 \end{split}
\end{equation}
where $\chi_u(x)=u\cdot x\cdot u^{-1}$. Note that for $u\in\OO$ with $\nup(u)\neq 0$, we have
\[
\chi_u(v)=\begin{cases} -v&\text{if $\nup(u,v)=0$,}\\
                                  u&\text{for $v=u$,}
                                  \end{cases}
\]
so that $\chi_u$ is minus the reflection relative to the hyperplane orthogonal to $u$.

\begin{remark}\label{re:SpinNatural}
Note that the projection $\pi_1:\Spin(\OO,\nup)\rightarrow \mathrm{SO}(\OO,\nup)$, $u\mapsto \chi_u$,
gives the natural representation of $\Spin(\OO,\nup)$ on $\OO$, while the projections 
$\pi_2:\Spin(\OO,\nup)\rightarrow \mathrm{SO}(\OO,\nup)$, $u\mapsto \rho^+_u$, and
$\pi_3:\Spin(\OO,\nup)\rightarrow \mathrm{SO}(\OO,\nup)$, $u\mapsto \rho^-_u$, give the two
half-spin representations.
\end{remark}

\begin{lemma}[{\cite[Corollary 5.6]{EKmon}}]\label{le:spin}
The map
\begin{equation}\label{eq:Psi}
\begin{split}
\Psi:\Spin(\OO,\nup)&\longrightarrow \Aut(\alb)\\
u\quad &\mapsto\quad \Psi_u: E_i\mapsto E_i,\ \iota_1(a)\mapsto\iota_1(\chi_u(a)),\\
   &\qquad\qquad\  \iota_2(a)\mapsto\iota_2(\rho^+_u(a)),\ 
     \iota_3(a)\mapsto\iota_3(\rho^-_u(a)),
\end{split}
\end{equation}
is a one-to-one group homomorphism with image the stabilizer in $\Aut(\alb)$ of the idempotents $E_i$, $i=1,2,3$.
\end{lemma}

\bigskip

\section{A suitable order \texorpdfstring{$5$}{5} automorphism of \texorpdfstring{$\alb$}{A}}\label{se:order5}

\emph{From now on, the characteristic of the ground field will be assumed to be $5$.}

\smallskip

For $i=0,\ldots,6$, consider the norm one element $u_i\bydef \frac{1}{\sqrt{2}}(e_i-e_{i+1})$ in $\OO$.
Then the reflection  $-\chi_{u_i}$ permutes $e_i$ and $e_{i+1}$ and fixes $e_j$ for $j\neq i,i+1$.

Consider then the element 
\begin{equation}\label{eq:u}
u=-u_3\cdot u_4\cdot u_5\cdot u_6,
\end{equation} 
which lies in $\Spin(\OO,\nup)$. The
rotation $\chi_u$ permutes cyclically 
\[
e_3\mapsto e_4\mapsto e_5\mapsto e_6\mapsto e_7\mapsto e_3
\]
and fixes $e_0$, $e_1$,  $e_2$. Equations \eqref{eq:Phi1} and \eqref{eq:Phi2} give, using $-\frac{1}{4}=1$, the following expressions:
\begin{equation}\label{eq:rho}
\begin{split}
\rho^-_u&=-l_{u_3}r_{u_4}l_{u_5}r_{u_6}=l_{e_3-e_4}r_{e_4-e_5}l_{e_5-e_6}r_{e_6-e_7},\\
\rho^+_u&=-r_{u_3}l_{u_4}r_{u_5}l_{u_6}=r_{e_3-e_4}l_{e_4-e_5}r_{e_5-e_6}l_{e_6-e_7}.
\end{split}
\end{equation}

 A straightforward but tedious computation gives $(\rho^{\pm}_u-\id)^4=0$, and the Jordan canonical
 form of $\rho^{\pm}_u$ consists of two blocks of length $4$.
 Note 
  that for $x,y\in\OO$, $r_x(y)=y\bullet x=\bar y\bar x
 =\overline{xy}=\overline{\bar x\bullet\bar y}=\nu l_{\bar x}\nu$, and hence we get 
 $\rho^+_u=\nu\rho^-_u\nu$, so it is enough to work with $\rho^-_u$. (The SageMath code to check 
the assertion on the Jordan blocks is given in  Appendix \ref{se:App2}.)

The conclusion is that the element $u$ in \eqref{eq:u} and the automorphism 
\begin{equation}\label{eq:psiauto}
\sigma\bydef\Psi_u
\end{equation} 
in \eqref{eq:Psi} have order $5$. The automorphism $\sigma$ 
decomposes $\alb$ into the direct sum of:
\begin{itemize}
\item Three Jordan blocks of length $1$: $\FF E_1$, $\FF E_2$, $\FF E_3$.
\item Another three Jordan blocks of length $1$ in $\iota_1(\OO)$: $\FF \iota_1(e_i)$, $i=0,1,2$.
\item A single Jordan block of length $5$ in $\iota_1(\OO)$: $\bigoplus_{i=3}^7\FF\iota_1(e_i)$.
\item Two Jordan blocks of length $4$ in $\iota_2(\OO)$: 
   $\bigoplus_{i=0}^3\FF\iota_2\bigl((\rho_u^+)^i(e_0)\bigr)$ and \break
   $\bigoplus_{i=0}^3\FF\iota_2\bigl((\rho_u^+)^i(e_2)\bigr)$.
\item Another two Jordan blocks of length $4$ in $\iota_3(\OO)$:
$\bigoplus_{i=0}^3\FF\iota_3\bigl((\rho_u^-)^i(e_0)\bigr)$ and
   $\bigoplus_{i=0}^3\FF\iota_3\bigl((\rho_u^-)^i(e_2)\bigr)$.
\end{itemize}

\begin{remark} If one takes the element $-u=u_3\cdot u_4\cdot u_5\cdot u_6$ instead of $u$, then the 
automorphism $\Psi_{-u}$ has order $10$, instead of $5$, as $\chi_u=\chi_{-u}$, but 
$\rho^{\pm}_{-u}=-\rho^{\pm}_u$.
\end{remark}

\begin{remark}\label{re:spin5}
Under the identification in \eqref{eq:Spin8RelatedTriples} of $\Spin(\OO,\nup)$ with
$ \{(\varphi_1,\varphi_2,\varphi_3)\in \mathrm{SO}(\OO,\nup)^3\mid \varphi_1(x\bullet y)=\varphi_2(x)\bullet \varphi_3(y)\ \forall x,y\in\OO\}$,
the automorphism $\sigma$ in \eqref{eq:psiauto} belongs to the image under $\Psi$ in 
\eqref{eq:Psi} of the subgroup $\{(\varphi_1,\varphi_2,\varphi_3)\in\Spin(\OO,\nup)\mid 
\varphi_1(e_i)=e_1\ \forall i=0,1,2\}$, which is isomorphic to $\Spin(5)$. The natural module of 
$\Spin(5)$ can be identified with $\bigoplus_{i=3}^{7}\FF e_i$, and the action of $\sigma$
 on this natural module is indecomposable (a unique
 Jordan block of length $5$). On the other hand, each of the half-spin modules for
 $\Spin(8)=\Spin(\OO,\nup)$ decomposes as the direct sum of two copies of the spin module for
 $\Spin(5)$. Hence the action of $\sigma\in\Spin(5)$ on this spin module is indecomposable again (unique Jordan block of length $4$).
 
 This $\Spin(5)$ subgroup of $\Aut(\alb)$ corresponds to the two middle vertices in Dynkin's diagram 
 of $F_4$: 
 \begin{picture}(75,10)
 \multiput(10,3)(20,0){4}{\circle{4}}
\put(12,3){\line(1,0){16}}
\put(52,3){\line(1,0){16}}
\put(31,4){\line(1,0){18}}
\put(31,2){\line(1,0){18}}
\put(36,.5){$>$}
\end{picture}.
\end{remark}

\bigskip

\section{From the Albert algebra to Kac's Jordan superalgebra \texorpdfstring{$K_{10}$}{K10}}

Consider the Albert algebra $\alb$ and its automorphism $\sigma=\Psi_u$ above, which allows us to
view $\alb$ as a Jordan algebra in $\Repu$. Then 
$\alb=\alb_1\oplus\alb_4\oplus\alb_5$ as in \S \ref{se:Repu}, where $\alb_i$ is a direct sum of copies
of $L_i$, and we may take the subspaces
$A\subo$ and $A\subuno$ as follows:
\begin{equation}\label{eq:A0A1}
\begin{split}
A\subo&=\FF E_1\oplus\FF E_2\oplus\FF E_3\oplus \FF \iota_1(e_0)\oplus\FF \iota_1(e_1)\oplus
                \FF \iota_1(e_2),\\
                &=\FF E_1\oplus\left(\FF E_2\oplus\FF E_3\oplus \FF \iota_1(e_0)\oplus\FF \iota_1(e_1)\oplus
                \FF \iota_1(e_2)\right)\\
                &\quad\text{(direct sum of two ideals),}\\[4pt]
A\subuno&=\FF\iota_2(e_0)\oplus\FF\iota_2(e_2)\oplus\FF\iota_3(e_0)\oplus\FF\iota_3(e_2).
\end{split}
\end{equation}
In this case, $A\subo$ is a subalgebra of $\alb$, so the multiplication $\diamond$ in Recipe \ref{recipe} coincides with the multiplication in $\alb$.
The Jordan algebra $A\subo$ is the direct sum of two simple ideals: $\FF E_0\cong \FF$ and
$\cI=\FF (E_2+E_3)\oplus V$, with
$V=\FF(E_2-E_3)\oplus\FF \iota_1(e_0)\oplus\FF \iota_1(e_1)\oplus\FF \iota_1(e_2)$, which is
the  Jordan algebra of the quadratic form $q$ on the vector space $V$, where $E_2+E_3$ is the
unity element of $\cI$, and  $(E_2-E_3)^2=E_2+E_3$, $\iota_1(e_i)^2=4(E_2+E_3)$, $i=0,1,2$, and all other products are $0$. Here $q$ is the quadratic form on $V$ where $E_2-E_3$, $\iota_1(e_0)$, 
$\iota_1(e_1)$, and $\iota_1(e_2)$ are orthogonal, and $q(E_2-E_3)=1$, $q\bigl(\iota_1(e_i)\bigr)=4$, $i=0,1,2$. Therefore, $A\subo$ is isomorphic to the even part of Kac's Jordan superalgebra $K_{10}$.

In order to check that the Jordan superalgebra $A\subo\oplus A\subuno$, with the multiplication in Recipe \eqref{recipe} is isomorphic to $K_{10}$, we can explicitly compute this multiplication for the elements
in the basis in \eqref{eq:A0A1}, but this is tedious. Alternatively, some known results can be used.

\begin{lemma}
Any unital Jordan bimodule for $A\subo$ such that both $E_1$ and $E_2+E_3$ act faithfully is a direct sum
of copies of a unique irreducible four-dimensional $A_0$-module.
\end{lemma}
\begin{proof} \cite[Chapter II, Theorem 16]{JacobsonJordan} shows that any such unital Jordan bimodule
is an associative module for the unital special universal envelope of the ideal $\cI$, which is the Clifford algebra of $(V,q)$, and hence isomorphic to $\Mat_4(\FF)$. The result follows.
\end{proof}

\begin{lemma}[{\cite[Lemma 3]{RZ}}] If $\cJ$ is a simple Jordan superalgebra with even part isomorphic
to $(K_{10})\subo$ and the odd part is the four-dimensional irreducible module in the previous lemma, 
then $\cJ$ is isomorphic to $K_{10}$.
\end{lemma}

The main result is an easy consequence.

\begin{theorem}\label{th:K10}
The Jordan superalgebra $(A\subo\oplus A\subuno,\diamond)$, obtained by semisimplification of the
Albert algebra, is isomorphic to Kac's ten-dimensional simple Jordan superalgebra $K_{10}$.
\end{theorem}
\begin{proof}
Note that we have
$E_1\diamond X=\frac{1}{2}X=(E_2+E_3)\diamond X$ for any $X\in A\subuno$. Hence,
by the previous two lemmas, in order to check that the Jordan superalgebra $(A\subo\oplus A\subuno,\diamond)$ above is isomorphic to $K_{10}$, it is enough to check that it is simple. Since $A\subuno$ is an irreducible module for $A\subo$, any nonzero ideal of  $(A\subo\oplus A\subuno,\diamond)$ contains $A\subuno$, and hence it is enough to show
that $A\subuno\diamond A\subuno$ is not contained in any of the two proper ideals of $A\subo$.
But a straightforward computation (see Appendix \ref{se:App2}) gives:
\begin{equation}\label{eq:iota3e2}
\begin{split}
\iota_3(e_0)\diamond\iota_3(e_2)&=
         \mathrm{proj}_{A\subo}\bigl(\iota_3(e_0)\cdot\delta^3(\iota_3(e_2))\bigr)\\
         &=\mathrm{proj}_{A\subo}\bigl(\iota_3(e_0)\cdot\iota_3(-2e_0+e_1+2e_3-2e_5+e_6-e_7)\bigr)\\
         &=-8(E_1+E_2)=2E_1+2E_2,
\end{split}
\end{equation}
which is not contained neither in $\FF E_1$ nor in $\cI$.
\end{proof}

\begin{remark}\label{re:Z22K10}
Our order $5$ automorphism $\sigma$ preserves the $\bigl(\ZZ/2\bigr)^2$-grading on $\alb$ in 
\eqref{eq:Z22grading}. As a consequence the Jordan superalgebra $A=A\subo\oplus A\subuno$ inherits
this grading:
\[
\begin{aligned}
A_{(\bar 0,\bar 0)}&=\FF E_1\oplus\FF E_2\oplus\FF E_3, &
A_{(\bar 0,\bar 1)}&=\FF \iota_1(e_0)\oplus\FF \iota_1(e_1)\oplus \FF \iota_1(e_2), \\
A_{(\bar 1,\bar 0)}&=\FF\iota_2(e_0)\oplus\FF\iota_2(e_2), &
A_{(\bar 1,\bar 1)}&=\FF\iota_3(e_0)\oplus\FF\iota_3(e_2).
\end{aligned}
\]
\end{remark}

\begin{remark}\label{re:bizarre}
The Albert algebra $\alb$ satisfies the Cayley-Hamilton equation $ch_3(x)=0$, with
\[
ch_3(x)=x^3-3t_{\alb}(x)x^2+\Bigl(\frac{9}{2}t_{\alb}(x)^2-\frac{3}{2}t_{\alb}(x^2)\Bigr)x-
 \Bigl(t_{\alb}(x^3)-\frac{9}{2}t_{\alb}(x^2)t_{\alb}(x)+\frac{9}{2}t_{\alb}(x)^3\Bigr)1,
\]
where $t_{\alb}(x)=\frac{1}{3}\trace(x)$ and $\trace$ denotes the usual trace on $\alb=H_3(\OO)$. 
The linear form $t_{\alb}$ is a \emph{normalized trace}: $t_{\alb}(1)=1$ ad
$t_{\alb}\bigl((x\cdot y)\cdot z\bigr)=t_{\alb}\bigl(x\cdot (y\cdot z)\bigr)$ for all $x,y,z\in\alb$.

It is clear that $t_{\alb}:\alb\rightarrow \FF$ is a morphism in $\Repu$ and the endomorphism
$\alb\xrightarrow{t_{\alb}}\FF \hookrightarrow \alb$ is idempotent. It then follows that our
Jordan superalgebra $(A=A\subo\oplus A\subuno,\diamond)$ is endowed with the normalized trace 
induced by $t_{\alb}$, which will be denoted by $t_A$. The linearized version of the
Cayley-Hamilton equation $ch_3(x)=0$ depends on $t_{\alb}$ and the multiplication in $\alb$, 
and hence, after semisimplification, our Jordan superalgebra $A$ also satisfies the induced
super version, with $t_{\alb}$ replaced by $t_A$.

In other words, the ``bizarre result'' proved by McCrimmon \cite{McC05} asserting that, in 
characteristic $5$, Kac's superalgebra satisfies the super version of $ch_3(x)=0$ follows from the 
fact that $K_{10}$ is obtained as a semisimplification of the Albert algebra $\alb$.
\end{remark}


\bigskip

\section{From \texorpdfstring{$F_4$}{F4} to 
\texorpdfstring{$\frosp_{1,2}\oplus\frosp_{1,2}$}{osp12+osp12}}\label{se:F4osp}

The Lie algebra of derivations $\Der(\alb)$ of the Albert algebra $\alb$ is the simple Lie algebra
of type $F_4$. Let us describe it in a way suitable for our purposes following \cite[\S 5.5]{EKmon}.

To begin with, consider the \emph{triality Lie algebra}:
\[
\tri(\OO)\bydef\{(d_1,d_2,d_3)\in \frso(\OO,\nup)^3\mid d_1(x\bullet y)=d_2(x)\bullet y+x\bullet d_3(y)\ \forall x,y\in\OO\},
\]
with component-wise bracket. Any of the three projections
\begin{equation}\label{eq:pitri}
\begin{split}
\pi_i:\tri(\OO)&\longrightarrow \frso(\OO,\nup)\\
     (d_1,d_2,d_3)&\mapsto\ d_i
\end{split}
\end{equation}
is an isomorphism of Lie algebras (see, e.g., \cite[Proposition 5.30]{EKmon}).

For any $(d_1,d_2,d_3)\in \tri(\OO)$ there is a derivation $D_{(d_1,d_2,d_3)}\in\Der(\alb)$ (compare to
Lemma \ref{le:spin}) given by:
\[
D_{(d_1,d_2,d_3)}:\ E_i\mapsto 0,\ \iota_i(x)\mapsto \iota_i\bigl(d_i(x)\bigr),
\]
for any $i=1,2,3$ and $x\in\OO$.

Also, for $x\in\OO$ and $i=1,2,3$, there is a derivation $D_i(x)\bydef 2[L_{\iota_i(x)},L_{E_{i+1}}]$, where the indices are taken modulo $3$ and $L_z$ denotes the multiplication by $z$ in $\alb$. The action
of $D_i(x)$ works as follows:
\[
\begin{split}
D_i(x):\quad &E_i\mapsto 0,\ E_{i+1}\mapsto \frac{1}{2}\iota_i(x),\ 
                                            E_{i+2}\mapsto -\frac{1}{2}\iota_i(x),\\
                  &\iota_i(y)\mapsto 2\nup(x,y)\bigl(-E_{i+1}+E_{i+2}\bigr),\\
                  &\iota_{i+1}(y)\mapsto -\iota_{i+2}(x\bullet y),\\
                  &\iota_{i+2}(y)\mapsto \iota_{i+1}(y\bullet x),
\end{split}
\]
for any $y\in \OO$. Denote by $D_i(\OO)$ the linear span of the $D_i(x)$'s.

The Lie algebra $\Der(\alb)$ is endowed with a natural $\bigl(\ZZ/2\bigr)^2$-grading, inherited
from the grading in \eqref{eq:Z22grading}, with
\begin{equation}\label{eq:Z22grading2}
\begin{aligned}
\Der(\alb)_{(\bar 0,\bar 0)}&=D_{\tri(\OO)},&
\Der(\alb)_{(\bar 0,\bar 1)}&=D_1(\OO),\\ 
\Der(\alb)_{(\bar 1,\bar 0)}&=D_2(\OO),&
\Der(\alb)_{(\bar 1,\bar 1)}&=D_3(\OO).
\end{aligned}
\end{equation}

Our order $5$ automorphism $\sigma$ in \eqref{eq:psiauto} induces, by conjugation, an order $5$
automorphism $\Ad_\sigma$ on $\Der(\alb)$, that allows to consider $\Der(\alb)$ as an algebra
in $\Repu$, whose action is the following:
\[
\begin{split}
&\Ad_\sigma\bigl(D_{(d_1,d_2,d_3)}\bigr)
 =D_{(\sigma_1d_1\sigma_1^{-1},\sigma_2d_2\sigma_2^{-1},\sigma_3d_3\sigma_3^{-1})}\\
&\Ad_\sigma\bigl(D_i(x)\bigr)=D_i\bigl(\sigma_i(x)\bigr) ,
\end{split}
\]
for any $(d_1,d_2,d_3)\in\tri(\OO)$, $i=1,2,3$, and $x\in\OO$, with 
$(\sigma_1,\sigma_2,\sigma_3)=(\chi_u,\rho_u^+,\rho_u^-)$.

The action of $\Ad_\sigma$ on $D_{\tri(\OO)}$ is determined by the action of $\sigma_1=\chi_u$
in $\frso(\OO,\nup)$, as the projection $\pi_1$ in \eqref{eq:pitri} is an isomorphism. Recall from
Section \ref{se:order5} that $\sigma_1$ has three Jordan blocks of length one: $\FF e_0$, $\FF e_1$, and $\FF e_2$,
and a Jordan block of length $5$: $W=\bigoplus_{i=3}^7\FF e_i$. Write 
$V=\FF e_0\oplus\FF e_1\oplus \FF e_2$. With the natural linear isomorphism 
\[
\frso(\OO,\nup)=\frso(V\perp W,\nup)\simeq\frso(V,\nup)\oplus\frso(W,\nup)\oplus (V\otimes W)
\]
we check that $\sigma_1$ acts trivially, by conjugation, on $\frso(V,\nup)$, and with two (respectively
three) Jordan blocks of length $5$ on $\frso(W,\nup)$ (resp. $V\otimes W$). 

Hence the order $5$ automorphism $\Ad_\sigma$ decomposes $\Der(\alb)$ into the direct sum of:
\begin{itemize}
\item Three Jordan blocks of length $1$, forming a subalgebra isomorphic to $\frso_3$,
and five Jordan blocks of length $5$ in $D_{\tri(\OO)}$.
\item Three Jordan blocks of length $1$: $\FF D_1(e_i)$, $i=0,1,2$, and one Jordan block of length $5$
in $D_1(\OO)$.
\item Two Jordan blocks of length $4$ in $D_2(\OO)$: 
$\bigoplus_{j=0}^3\FF D_2\bigl(\sigma_2^j(e_0)\bigr)$ and\\
$\bigoplus_{j=0}^3\FF D_2\bigl(\sigma_2^j(e_2)\bigr)$.
\item Another two Jordan blocks of length $4$ in $D_3(\OO)$: 
$\bigoplus_{j=0}^3\FF D_3\bigl(\sigma_3^j(e_0)\bigr)$ and
$\bigoplus_{j=0}^3\FF D_3\bigl(\sigma_3^j(e_2)\bigr)$.
\end{itemize}

Therefore $\cD\bydef\Der(\alb)$ decomposes, under the action of $\Ad_\sigma$ as 
$\cD=\cD_1\oplus\cD_4\oplus\cD_5$, which can be refined to
\begin{equation}\label{eq:Dspliting}
\cD=D\subo\oplus D\subuno\oplus \Delta(\cD_4)\oplus \cD_5
\end{equation}
as in \eqref{eq:Csplitting}, where $\Delta=\Ad_\sigma -\id$, with
\begin{equation}\label{eq:D0D1}
\begin{split}
&D\subo=\pi_1^{-1}\bigl(\frso(\FF e_0\oplus\FF e_1\oplus\FF e_2,\nup)\bigr)\oplus 
   \FF D_1(e_0)\oplus \FF D_1(e_1)\oplus \FF D_1(e_2),\\
&D\subuno= \FF D_2(e_0)\oplus \FF D_2(e_2)\oplus \FF D_3(e_0)\oplus\FF D_3(e_2).
\end{split}
\end{equation}
The simple Lie algebra $\cD=\Der(\alb)$, which is simple of type $F_4$, \emph{semisimplifies} to the
Lie superalgebra $D=D\subo\oplus D\subuno$, with bracket given by Recipe \ref{recipe}.

We can compute explicitly the bracket on $D$, but there is no need to do it. The action by derivations
$\mu:\Der(\alb)\otimes\alb\rightarrow \alb$ is a morphism in $\Repu$, which semisimplies to the action
$\Lambda: D\otimes A\rightarrow A$ by derivations in $\sVec$ given by Recipe \ref{recipe1}. This
gives a homomorphism of Lie superalgebras 
\begin{equation}\label{eq:Lambda}
\Lambda':D\rightarrow \Der(A),\quad d\mapsto \Lambda'(d):a\mapsto \Lambda(d\otimes a).
\end{equation}

In \cite{DES} there are examples where this homomorphism $\Lambda'$ fails to be an isomorphism.
However, here everything works smoothly.

\begin{proposition}\label{pr:LambdaIso}
The homomorphism of Lie superalgebras $\Lambda'$ in \eqref{eq:Lambda} is an isomorphism.
\end{proposition}
\begin{proof}
The Lie superalgebra $\Der(A)$ is, up to isomorphism, a direct sum of two copies of the
orthosymplectic Lie superalgebra $\frosp_{1,2}$ (see \cite[Theorem 2.8.(c)]{BE02}). In particular, it has dimension $10$, which is too
the dimension of $D=D\subo\oplus D\subuno$. Hence it is enough to check that $\Lambda'$ is one-to-one.

Notice that $\sigma$ preserves the $\bigl(\ZZ/2\bigr)^2$-grading of $\alb$ in \eqref{eq:Z22grading}, and
hence $\Ad_\sigma$ preserves the induced grading on $\Der(\alb)$ in \eqref{eq:Z22grading2}. The
morphism $\mu:\Der(\alb)\otimes\alb\rightarrow \alb$ in $\Repu$ is homogeneous of degree 
$(\bar 0,\bar 0)$ and hence so is $\Lambda'$, where the $\bigl(\ZZ/2\bigr)^2$-grading on $D$, induced
from the grading in $\Der(\alb)$ is given as follows:
\[
\begin{aligned}
D_{(\bar 0,\bar 0)}&=\pi_1^{-1}\bigl((\frso(\FF e_0\oplus\FF e_1\oplus\FF e_2,\nup)\bigr), &
D_{(\bar 0,\bar 1)}&=\FF D_1(e_0)\oplus \FF D_1(e_1)\oplus \FF D_1(e_2),\\
D_{(\bar 1,\bar 0)}&=\FF D_2(e_0)\oplus \FF D_2(e_2),&
D_{(\bar 1,\bar 1)}&=\FF D_3(e_0)\oplus \FF D_3(e_2).
\end{aligned}
\]
Recipe \ref{recipe1} gives the following:
\begin{itemize}
\item 
$\Lambda\bigl(\pi_1^{-1}(d)\otimes \iota_1(v)\bigr)=\iota_1\bigl(d(v)\bigr)$, for any 
$d\in\frso(\FF e_0\oplus\FF e_1\oplus\FF e_2,\nup)$ and $v\in \FF e_0\oplus\FF e_1\oplus \FF e_2$, and this shows that 
$\Lambda'$ is one-to-one on $D_{(\bar 0,\bar 0)}$,
\item 
$\Lambda\bigl(D_1(e_i)\otimes E_2\bigr)=\frac{1}{2}\iota_1(e_i)$, for any $i=0,1,2$, and thus 
$\Lambda'$ is one-to-one on $D_{(\bar 0,\bar 1)}$,
\item
$\Lambda\bigl(D_2(e_i)\otimes E_3\bigr)=\frac{1}{2}\iota_2(e_i)$, for $i=0,2$, so that 
$\Lambda'$ is one-to-one on $D_{(\bar 1,\bar 0)}$,
\item
$\Lambda\bigl(D_3(e_i)\otimes E_1\bigr)=\frac{1}{2}\iota_3(e_i)$, for $i=0,2$, so that 
$\Lambda'$ is one-to-one on $D_{(\bar 1,\bar 0)}$.
\end{itemize}
We conclude that $\Lambda'$ is an isomorphism.
\end{proof}

Proposition \ref{pr:LambdaIso} tells us that the semisimplification process converts the simple
Lie algebra $\Der(\alb)$ of type $F_4$ into the semisimple Lie superalgebra 
$\frosp_{1,2}\oplus\frosp_{1,2}$.

\bigskip

\section{From \texorpdfstring{$E_8$}{E8} to 
\texorpdfstring{$\frel(5;5)$}{el(5;5)}}\label{se:E8el55}

The simple Lie superalgebra $\frel(5;5)$ is specific of characteristic $5$. It appeared in 
\cite{Eld07}. Its even part is the orthogonal Lie algebra $\frso_{11}$ and its odd part is the
spin module for the even part. It was remarked in \cite{Eld07}, and this was the reason of its
discovery, that it can be obtained from the well-known Tits construction in \cite{Tits66} (see
also \cite{Eld11} and the references therein).

The simple Lie algebra of type $E_8$ appears in this construction as
\[
\cT(\OO,\alb)\bydef \Der(\OO)\oplus \bigl(\OO^0\otimes \alb^0\bigr)\oplus \Der(\alb),
\]
with $\OO^0\bydef\{x\in\OO\mid \nup(1,x)=0\}$, $\alb^0\bydef\{x\in\alb\mid \trace(x)=0\}$,
where $\trace$ denotes the usual trace in $\alb=H_3(\OO)$. The bracket in $\cT(\OO,\alb)$ is
determined as follows:
\begin{itemize}
\item The Lie algebras $\Der(\OO)$ and $\Der(\alb)$ are subalgebras of $\cT(\OO,\alb)$, and we have
$[\Der(\OO),\Der(\alb)]=0$,
\item
$[D,a\otimes x]=D(a)\otimes x$, $[d,a\otimes x]=a\otimes d(x)$, for any $D\in\Der(\OO)$,
$d\in\Der(\alb)$, and $a\in\OO^0$, $x\in\alb^0$,
\item
$[a\otimes x,b\otimes y]=\frac{1}{3}\trace(x\cdot y)D_{a,b}+
[a,b]\otimes (x\cdot y-\frac{1}{3}\trace(x\cdot y)1)-2\nup(a,b)d_{x,y}$, for $a,b\in\OO^0$ and 
$x,y\in \alb^0$, where $D_{a,b}$ is the derivation of $\OO$ given by 
$D_{a,b}(c)=[[a,b],c]+3((ac)b-a(cb))$, while $d_{x,y}$ is the derivation of $\alb$ given by
$d_{x,y}(z)=x\cdot (y\cdot z)-y\cdot (x\cdot z)$.
\end{itemize}

Our order $5$ automorphism $\sigma$ on $\alb$ preserves $\alb^0=\FF(E_1-E_2)\oplus\FF(E_2-E_3)\oplus \iota_1(\OO)\oplus\iota_2(\OO)\oplus\iota_3(\OO)$, and hence it induces an order $5$
automorphism of $\cT(\OO,\alb)$ through its action on $\alb^0$ and, by conjugation, on 
$\Der(\alb)$. This allows to view $\cT(\OO,\alb)$ as a Lie algebra in $\Repu$. Collecting previous
information, the decomposition of $\cT(\OO,\alb)$ into indecomposable modules for the
action of $\sigma$ is $\cT(\OO,\alb)=(14+7\times 5+6)L_1+(7\times 4+4)L_4+(7\times 1+6)L_5
=55 L_1+ 32 L_4 + 13 L_5$.

The Lie superalgebra obtained by semisimplification is defined, after the natural identification of
$D=D\subo\oplus D\subuno$ with $\Der(A)$ by means of the isomorphism $\Lambda'$ in 
Proposition \ref{pr:LambdaIso}, on the vector space
\[
\Der(\OO)\oplus \bigl(\OO^0\otimes A^0\bigr)\oplus \Der(A)
\]
where $A=A\subo\oplus A\subuno$ is the Jordan superalgebra, isomorphic to $K_{10}$, obtained 
from $\alb$ (Theorem \ref{th:K10}). The linear form $t_{\alb}=\frac{1}{3}\trace$ on $\alb$
in Remark \ref{re:bizarre} induces a 
normalized trace form $t$ on $A$, that is, $t(1)=1$, $t(A\subuno)=0$, and 
$t\bigl((xy)z\bigr)=t\bigl(x(yz)\bigr)$, which is necessarily unique, and $A^0$ is the subspace
$\{x\in A\mid t(x)=0\}$. The Lie superalgebra thus obtained, with the bracket given by Recipe 
\ref{recipe}, is then the Lie superalgebra $\cT(\OO,A)$ obtained by the super version of Tits
construction (see \cite{Eld11} and the references therein), and this is, up to isomorphism, the
exceptional Lie superalgebra $\frel(5;5)$.

\smallskip

\begin{remark}\label{re:F(4)}
The simple Lie algebra of type $E_7$ appears in Tits construction as $\cT(\HH,\alb)$, where $\HH$ is
the quaternion algebra over $\FF$, i.e., the associative algebra of $2\times 2$-matrices. After 
semisimplification we get $\cT(\HH,K_{10})$, which is the simple Lie superalgebra $F(4)$. In the same vein, the Lie algebra $\cT(\FF\oplus\FF,\alb)$ of type $E_6$ 
semisimplifies to $\cT(\FF\oplus\FF,K_{10})$, which is the orthosymplectic Lie superalgebra
$\frosp_{2,4}$, and the Lie algebra $\cT(\FF,\alb)$ of type $F_4$ to 
$\cT(\FF,K_{10})\cong\Der(K_{10})\cong\frosp_{1,2}\oplus\frosp_{1,2}$  (see 
\cite[Table 2]{Eld11}).
\end{remark}


\bigskip\bigskip

\appendix

\section{}\label{se:App1}

In this appendix, a canonical version of Recipe \ref{recipe} is obtained. In a similar vein, one
gets a canonical version of Recipe \ref{recipe1}.

Let $(\cA,\mu)$ be an algebra in $\Repp$, with $\mu(x\otimes y)=xy$ for all $x,y$. Fix a decomposition
$\cA=A\subo\oplus \cA_2\oplus\cdots\cA_{p-2}\oplus A\subuno\oplus \delta(\cA_{p-1})\oplus\cA_p$ 
as in \eqref{eq:Csplitting}, and recall the  multiplication $m$ in $A=A\subo\oplus A\subuno$ by means of Recipe \ref{recipe}: $m(x\otimes y)=x\diamond y$. As usual, write $\delta=\sigma-1$, which satisfies 
$\delta(xy)=\delta(x)y+x\delta(y)+\delta(x)\delta(y)$ for all $x,y\in\cA$.

Following \cite[Section 3.1]{EtOs20}, consider the vector spaces
\[
\tilde A\subo\bydef \ker\delta/(\ker\delta\cap\im\delta),\quad
\tilde A\subuno\bydef (\ker\delta\cap\im\delta^{p-2})/\im\delta^{p-1},
\]
and their direct sum $\tilde A=\tilde A\subo\oplus\tilde A\subuno$. For any $x\in\ker\delta$, denote by
$[x]\subo$ the class of $x$ modulo $\ker\delta\cap\im\delta$, and for any 
$x\in\ker\delta\cap\im\delta^{p-2}$, denote by $[x]\subuno$ the class of $x$ modulo $\im\delta^{p-1}$.

The linear map $\phi:A=A\subo\oplus A\subuno\rightarrow \tilde A$ given by
\begin{equation}\label{eq:phi}
\phi(x)=[x]\subo\quad \forall x\in A\subo,\qquad \phi(x)=[\delta^{p-2}(x)]\subuno\quad 
\forall x\in A\subuno,
\end{equation}
is a linear isomorphism.

\begin{theorem}
Define a multiplication on $\tilde A$ as follows:
\[
\begin{split}
[x]\subo\star [y]\subo&= [xy]\subo,\ \forall x,y\in\ker\delta,\\
[x]\subo\star [y]\subuno&=[xy]\subuno,\ \forall x\in\ker\delta,\ \forall y\in\ker\delta\cap\im\delta^{p-2},\\
[x]\subuno\star [y]\subo&=[xy]\subuno,\ \forall x\in\ker\delta\cap\im\delta^{p-2},\ \forall y\in\ker\delta,\\
[x]\subuno\star [y]\subuno&=\sum_{i=0}^{p-2}(-1)^i\delta^{p-2-i}(a)\delta^i(b)
    +\sum_{\substack{0\leq i,j\leq p,\\ i+j\geq p-1}}\mu_{ij}\delta^i(a)\delta^j(b) \\
    &\qquad\qquad\qquad\qquad 
   \forall x=\delta^{p-2}(a), y=\delta^{p-2}(b) \in \ker\delta\cap\im\delta^{p-2}.
\end{split}
\]
Then, this multiplication $\star$ is well defined, and the linear isomorphism $\phi$ in \eqref{eq:phi} is an isomorphism of
algebras: $\phi(x\diamond y)=\phi(x)\star \phi(y)$ for all $x,y\in A\subo\cup A\subuno$.
\end{theorem}

\begin{proof}
For any $x,y\in\ker\delta$, $xy\in\ker\delta$ holds, and hence we get the bilinear map
\[
\begin{split}
\ker\delta\times\ker\delta&\longrightarrow \tilde A\subo=\ker\delta/(\ker\delta\cap\im\delta)\\
(x,y)\ &\mapsto\  [xy]\subo.
\end{split}
\]
For $x\in\ker\delta\cap\im\delta$ and $y\in\ker\delta$, we have $x=\delta(u)$ for some $u$, and 
$xy=\delta(u)y=\delta(uy)\in\im\delta$, so $[xy]\subo=0$ holds. In the same vein, for $x\in\ker\delta$ and 
$y\in\ker\delta\cap\im\delta$, we get $xy\in\im\delta$. Hence, the bilinear map above gives the bilinear 
map $\tilde A\subo\times\tilde A\subo\rightarrow \tilde A\subo$ given by $\left([x]\subo,[y]\subo\right)
\mapsto [xy]\subo$.

For $x$ in $\ker\delta$ and $y\in\ker\delta\cap\im\delta^{p-2}$, we have $y=\delta^{p-2}(u)$ for some
$u$, and $xy=x\delta^{p-2}(u)=\delta^{p-2}(xu)$ lies in $\ker\delta\cap\im\delta^{p-2}$. Hence,
we get the bilinear map
\[
\begin{split}
\ker\delta\times(\ker\delta\cap\im\delta^{p-2})&\longrightarrow 
   \tilde A\subuno=(\ker\delta\cap\im\delta^{p-2})/\im\delta^{p-1}\\
(x,y)\ &\mapsto\  [xy]\subuno.
\end{split}
\]
For $x\in\ker\delta\cap\im\delta$ and $y\in\ker\delta\cap\im\delta^{p-2}$, $x=\delta(u)$ for some $u$
and $y=\delta^{p-2}(v)$ for some $v$. Hence, $xy=\delta(u)\delta^{p-2}(v)=\delta^{p-1}(uv) \in
\im\delta^{p-1}$ as $\delta^i(u)=0$ for $i\geq 2$ and $\delta^{p-1}(v)=0$. Thus, $[xy]\subuno=0$ holds. Also, for $x\in\ker\delta$ and $y\in\im\delta^{p-1}$, we have $y=\delta^{p-1}(v)$ and 
$xy=x\delta^{p-1}(v)=\delta^{p-1}(xv)\in\im\delta^{p-1}$. As a consequence, the above bilinear map
induces the bilinear map $\tilde A\subo\times \tilde A\subuno\rightarrow \tilde A\subuno$ given by
$\left([x]\subo,[y]\subuno\right)\mapsto [xy]\subuno$.

Similarly we obtain the bilinear map $\tilde A\subuno\times \tilde A\subo\rightarrow \tilde A\subuno$ 
given by $\left([x]\subuno,[y]\subo\right)\mapsto [xy]\subuno$.

Finally, for $x,y\in\ker\delta\cap\im\delta^{p-2}$, there are elements $a,b\in \cA$ such that
$x=\delta^{p-2}(a)$, $y=\delta^{p-2}(b)$, consider the element
\[
w(a,b)=\sum_{i=0}^{p-2}(-1)^i\delta^{p-2-i}(a)\delta^i(b)
    +\sum_{\substack{0\leq i,j\leq p,\\ i+j\geq p-1}}\mu_{ij}\delta^i(a)\delta^j(b).
\]
Note that $w(a,b)$ lies in $\ker\delta$, because it is the image of the element $w$ in \eqref{eq:w} under
the $\mathsf{C}_p$-invariant linear map $L_{p-1}\otimes L_{p-1}\rightarrow \cA$, $v_i\otimes v_j\mapsto \delta^i(a)\delta^j(b)$.

If $\delta^{p-2}(a)=\delta^{p-2}(a')$, and $a-a'=u$, we have $\delta^{p-2}(u)=0$ and
$w(a,b)-w(a',b)=\sum_{i=0}^{p-2}(-1)^i\delta^{p-2-i}(u)\delta^i(b)
    +\sum_{\substack{0\leq i,j\leq p,\\ i+j\geq p-1}}\mu_{ij}\delta^i(u)\delta^j(b)$. Then, with the 
arguments in the proof of Lemma \ref{le:Lp-1Lp-1L1} leading to \eqref{eq:Ad}, we get $\delta^i(u)\delta^j(b)\in\im\delta$ for
$i+j\geq p-2$. Therefore, the class of $w(a,b)$ modulo $\im\delta$ does not depend on the chosen element $a$ with $x=\delta^{p-2}(a)$, and the same for $b$. Hence, the bilinear map
\[
\begin{split}
(\ker\delta\cap\im\delta^{p-2})\times(\ker\delta\cap\im\delta^{p-2})&\longrightarrow 
   \tilde A\subo=\ker\delta/(\ker\delta\cap\im\delta)\\
\bigl(x=\delta^{p-2}(a),y=\delta^{p-2}(b)\bigr)&\mapsto [w(a,b)]\subo
\end{split}
\]
is well defined.  If $x$ lies in $\im\delta^{p-1}$, so $x=\delta^{p-2}(a)$ and $a=\delta(u)$ for some $u$, then the arguments in the proof  of Lemma \ref{le:Lp-1Lp-1L1} give $\delta^i(u)\delta^j(b)\in\im\delta$ for $i+j\geq p$, and this gives $w(a,b)\in\im\delta$. As a consequence, the above bilinear map induces the bilinear map 
$\tilde A\subuno\times \tilde A\subuno\rightarrow \tilde A\subo$
given by
\[
\left([x=\delta^{p-2}(a)]\subuno,[y=\delta^{p-2}(b)]\subuno\right)\, \mapsto\, [w(a,b)]\subo.
\]

We conclude that the multiplication $\star$ is well defined. Let us check that the linear bijection $\phi$ in \eqref{eq:phi}
is an algebra isomorphism $(A,\diamond)\rightarrow (\tilde A,\star)$.

For $x,y\in A\subo$, $\phi(x)\star \phi(y)=[x]\subo\star [y]\subo=[xy]\subo=[\mathrm{proj}_{A\subo}(xy)]\subo
=\phi(x\diamond y)$, because $xy-\mathrm{proj}_{A\subo}(xy)$ lies in $\ker\delta\cap\im\delta$.

For $x\in A\subo$ and $y\in A\subuno$, write $xy=z_1+\cdots z_p$ with $z_i\in\cA_i$ for all $i$. As $x$ lies in $\ker\delta$, $\delta^{p-1}(xy)=x\delta^{p-1}(y)=0$, so $\delta^{p-1}(z_p)=0$ and $z_p\in\delta(\cA_p)$. Hence, we get $\delta^{p-2}(xy)=\delta^{p-2}(z_{p-1})+\delta^{p-2}(z_p)=
\delta^{p-2}\bigl(\mathrm{proj}_{A\subuno}(xy)\bigr)+\delta^{p-2}(z_p)\in 
\delta^{p-2}\bigl(\mathrm{proj}_{A\subuno}(xy)\bigr)+\im\delta^{p-1}$, and we have
\begin{multline*}
\phi(x)\star\phi(y)=[x]\subo\star[\delta^{p-2}(y)]\subuno=[x\delta^{p-2}(y)]\subuno
\\
=[\delta^{p-2}(xy)]\subuno=
\bigl[\delta^{p-2}\bigl(\mathrm{proj}_{A\subuno}(xy)\bigr)\bigr]\subuno
=\phi(x\diamond y).
\end{multline*}
In the same vein, $\phi(x)\star\phi(y)=\phi(x\diamond y)$ for $x\in A\subuno$ and $y\in A\subo$.

And finally, for $x,y\in A\subuno$, we get
\[
\begin{split}
\phi(x)\star\phi(y)
    &=[\delta^{p-2}(x)]\subuno\star [\delta^{p-2}(y)]\subuno\\
    &=[w(x,y)]\subo =\bigl[\mathrm{proj}_{A\subo}\bigl(w(x,y)\bigr)\bigr]\subo\\
    &=\bigl[\mathrm{proj}_{A\subo}\bigl(x\delta^{p-2}(y)\bigr)\bigr]\subo\quad 
         \text{(using the arguments in \eqref{eq:wx1y1})}\\
    &=\phi(x\diamond y),
\end{split}
\]
as desired.
\end{proof}

Both the vector superspace $\tilde A$ and the multiplication $\star$ are defined in a way independent of
noncanonical decompositions.

\bigskip

\section{}\label{se:App2}

Here is the SageMath code used to obtain the Jordan blocks of $\rho^-_u$ in \eqref{eq:rho}, and later  on
the expression for $\delta^3(\iota_3(e_2))$ in \eqref{eq:iota3e2}:

\medskip

{\obeylines
\def\salto{\noindent\phantom{re7=matrix(In}}
\noindent\texttt{%
\# Matrices of $r_{e_7}\,$, $l_{e_6}\,$, $r_{e_6}\,$, $l_{e_5}\,$, $r_{e_5}\,$, $l_{e_4}\,$, $r_{e_4}\,$, and $l_{e_3}\,$:

re7=matrix(Integers(5),[[0,0,0,0,0,0,0,-1],[0,0,0,-1,0,0,0,0],
\salto      [0,0,0,0,0,0,-1,0],[0,1,0,0,0,0,0,0],
\salto      [0,0,0,0,0,-1,0,0],[0,0,0,0,1,0,0,0],
\salto      [0,0,1,0,0,0,0,0],[-1,0,0,0,0,0,0,0]]).transpose();
le6=matrix(Integers(5),[[0,0,0,0,0,0,-1,0],[0,0,0,0,0,1,0,0],
\salto            [0,0,0,0,0,0,0,-1],[0,0,0,0,1,0,0,0],
\salto            [0,0,0,-1,0,0,0,0],[0,-1,0,0,0,0,0,0],
\salto            [-1,0,0,0,0,0,0,0],[0,0,1,0,0,0,0,0]]).transpose();
re6=matrix(Integers(5),[[0,0,0,0,0,0,-1,0],[0,0,0,0,0,-1,0,0],
\salto            [0,0,0,0,0,0,0,1],[0,0,0,0,-1,0,0,0],
\salto            [0,0,0,1,0,0,0,0],[0,1,0,0,0,0,0,0],
\salto            [-1,0,0,0,0,0,0,0],[0,0,-1,0,0,0,0,0]]).transpose();
le5=matrix(Integers(5),[[0,0,0,0,0,-1,0,0],[0,0,0,0,0,0,-1,0],
 \salto           [0,0,0,1,0,0,0,0],[0,0,-1,0,0,0,0,0],
\salto            [0,0,0,0,0,0,0,-1],[-1,0,0,0,0,0,0,0],
\salto            [0,1,0,0,0,0,0,0],[0,0,0,0,1,0,0,0]]).transpose();
re5=matrix(Integers(5),[[0,0,0,0,0,-1,0,0],[0,0,0,0,0,0,1,0],
\salto            [0,0,0,-1,0,0,0,0],[0,0,1,0,0,0,0,0],
\salto            [0,0,0,0,0,0,0,1],[-1,0,0,0,0,0,0,0],
\salto            [0,-1,0,0,0,0,0,0],[0,0,0,0,-1,0,0,0]]).transpose();
le4=matrix(Integers(5),[[0,0,0,0,-1,0,0,0],[0,0,1,0,0,0,0,0],
\salto            [0,-1,0,0,0,0,0,0],[0,0,0,0,0,0,-1,0],
\salto            [-1,0,0,0,0,0,0,0],[0,0,0,0,0,0,0,1],
\salto            [0,0,0,1,0,0,0,0],[0,0,0,0,0,-1,0,0]]).transpose();
re4=matrix(Integers(5),[[0,0,0,0,-1,0,0,0],[0,0,-1,0,0,0,0,0],
\salto            [0,1,0,0,0,0,0,0],[0,0,0,0,0,0,1,0],
\salto            [-1,0,0,0,0,0,0,0],[0,0,0,0,0,0,0,-1],
\salto            [0,0,0,-1,0,0,0,0],[0,0,0,0,0,1,0,0]]).transpose();
le3=matrix(Integers(5),[[0,0,0,-1,0,0,0,0],[0,0,0,0,0,0,0,-1],
\salto            [0,0,0,0,0,-1,0,0],[-1,0,0,0,0,0,0,0],
\salto            [0,0,0,0,0,0,1,0],[0,0,1,0,0,0,0,0],
\salto            [0,0,0,0,-1,0,0,0],[0,1,0,0,0,0,0,0]]).transpose();

\null
A=(le3-le4)*(re4-re5)*(le5-le6)*(re6-re7);

\null
J=A.jordan\_form(); J

\null
\salto    [1 1 0 0|0 0 0 0]
\salto    [0 1 1 0|0 0 0 0]
\salto    [0 0 1 1|0 0 0 0]
\salto    [0 0 0 1|0 0 0 0]
\salto    [-------+-------]
\salto    [0 0 0 0|1 1 0 0]
\salto    [0 0 0 0|0 1 1 0]
\salto    [0 0 0 0|0 0 1 1]
\salto    [0 0 0 0|0 0 0 1]

\noindent\# Thus, A has exactly two Jordan blocks of length 4

\null
T = J.is\_similar(A, transformation=True); T

\salto    (
\salto\ \ \           [0 2 1 1 3 3 0 0]
\salto\ \ \           [0 0 0 0 1 1 3 0]
\salto\ \ \           [2 2 0 0 0 2 1 1]
\salto\ \ \           [3 0 0 0 2 4 2 0]
\salto\ \ \           [1 1 3 0 0 0 0 0]
\salto\ \ \           [3 1 3 0 3 0 0 0]
\salto\ \ \           [4 2 0 0 1 4 3 0]
\salto       True,\ [4 1 2 0 4 2 0 0]
\salto    )

\noindent\# Therefore, the 'heads' of these two Jordan blocks are $e_0$ and $e_2$ 
(fourth and eighth columns)

\null
\noindent\# Let us compute now $\delta^3(\iota_3(e_2))$:

(A-identity\_matrix(8))\^{}3*vector([0,0,1,0,0,0,0,0])

\null
\salto   (3, 1, 0, 2, 0, 3, 1, 4)

\null
\noindent\# Hence, $\delta^3(\iota_3(e_2))=\iota_3(-2e_0+e_1+2e_3-2e_5+e_6-e_7)$.
}}

\bigskip\bigskip\null

\end{document}